\documentclass[12pt]{amsart}

\usepackage{amsmath}
\usepackage{amssymb}
\usepackage{amscd}
\usepackage[all]{xy}
\usepackage[dvips]{graphicx}
\usepackage{longtable}

\newtheorem{theorem}[subsection]{Theorem}
\newtheorem{lemma}[subsection]{Lemma}
\newtheorem{corollary}[subsection]{Corollary}
\newtheorem{clm}[subsection]{Claim}

\newtheorem{definition}[subsection]{{\sc Definition}}
\newtheorem{example}[subsection]{{\sc Example}}

\newtheorem{remark}[subsection]{{\sc Remark}}
\numberwithin{equation}{section}

\makeatletter
\@addtoreset{subsection}{section}
\@addtoreset{subsubsection}{section}
\@addtoreset{subsubsection}{subsection}
\@addtoreset{equation}{section}
\makeatother

\makeatletter
\newcommand\testshape{family=\f@family; series=\f@series; shape=\f@shape.}
\def\myemphInternal#1{\if n\f@shape%
\begingroup\itshape #1\endgroup\/%
\else\begingroup\bfseries #1\endgroup%
\fi}
\def\myemph{\futurelet\testchar\MaybeOptArgmyemph}
\def\MaybeOptArgmyemph{\ifx[\testchar \let\next\OptArgmyemph
                 \else \let\next\NoOptArgmyemph \fi \next}
\def\OptArgmyemph[#1]#2{\index{#1}\myemphInternal{#2}}
\def\NoOptArgmyemph#1{\myemphInternal{#1}}
\makeatother

\newcommand\RRR{\mathbb{R}}
\newcommand\CCC{\mathbb{C}}
\newcommand\ZZZ{\mathbb{Z}}
\newcommand\NNN{\mathbb{N}}
\newcommand\FFF{\mathbb{F}}

\newcommand\LL{\mathcal{L}}
\newcommand\FF{\mathcal{F}}

\newcommand\XX{\mathcal{X}}
\newcommand\YY{\mathcal{Y}}

\newcommand\Int{\mathrm{Int}}
\newcommand\Fr{\mathrm{Fr}}
\newcommand\im{\mathrm{im\,}}
\newcommand\id{\mathrm{id}}
\newcommand\defeq{\stackrel{\mathrm{def}}{=\!=}}

\newcommand\Nz{\NNN_{0}}
\newcommand\Nzi{\overline{\NNN}_{0}}

\newcommand\AFld{F}
\newcommand\AFlow{\mathbf{\AFld}}
\newcommand\BFld{G}
\newcommand\BFlow{\mathbf{\BFld}}

\newcommand\AFlowW{\AFlow_{\Wman}}

\newcommand\BeFld{\bar\BFld}
\newcommand\BeFlow{\bar\BFlow}

\newcommand\Mman{M}
\newcommand\afunc{\alpha}
\newcommand\bfunc{\beta}
\newcommand\cfunc{\gamma}
\newcommand\amap{f}
\newcommand\bmap{g}
\newcommand\cmap{h}

\newcommand\amapB{\mathbf{f}}
\newcommand\bmapB{\mathbf{g}}

\newcommand\Dm{D}

\newcommand\Nman{N}

\newcommand\Vman{V}
\newcommand\Uman{U}
\newcommand\Wman{W}

\newcommand\domA{\mathsf{dom}(\AFlow)}

\newcommand\domAW{\mathsf{dom}(\AFlowW)}

\newcommand\domB{\mathsf{dom}(\BFlow)}
\newcommand\domBe{\mathsf{dom}(\BeFlow)}

\newcommand\funcA{\mathsf{func}(\AFlow)}
\newcommand\funcAV{\mathsf{func}(\AFlow,\Vman)}

\newcommand\funcAVi{\mathsf{func}(\AFlow,\Vman_i)}

\newcommand\funcAWV{\mathsf{func}(\AFlowW,\Vman)}

\newcommand\funcBV{\mathsf{func}(\BFlow,\Vman)}

\newcommand\funcBeV{\mathsf{func}(\BeFlow,\Vman)}

\newcommand\SectShA{\sigma}
\newcommand\SectShAV{\SectShA_{\Vman}}

\newcommand\ShA{\varphi}
\newcommand\ShAV{\ShA_{\Vman}}

\newcommand\ShAWV{\ShA_{\Wman,\Vman}}

\newcommand\ShAVi{\ShA_{\Vman_i}}

\newcommand\ShB{\psi}
\newcommand\ShBe{\bar\ShB}
\newcommand\ShBV{\ShB_{\Vman}}

\newcommand\ShBeV{\ShBe_{\Vman}}

\newcommand\ZidAV{\ker(\ShAV)}

\newcommand\ZidBeV{\ker(\ShBeV)}

\newcommand\Mbh{\mathcal{M}}
\newcommand\Nbh{\mathcal{N}}

\newcommand\Unbh{\mathcal{U}}
\newcommand\AMbh{\mathcal{M}}
\newcommand\ANbh{\mathcal{N}}

\newcommand\eps{\varepsilon}

\newcommand\EAFlow{\mathcal{E}(\AFlow)}
\newcommand\DAFlow{\mathcal{D}(\AFlow)}
\newcommand\EidAFlow[1]{\mathcal{E}_{\id}(\AFlow)^{#1}}
\newcommand\DidAFlow[1]{\mathcal{D}_{\id}(\AFlow)^{#1}}
\newcommand\DiffM{\mathcal{D}(\Mman)}
\newcommand\EAFlowpr{\mathcal{E}(\AFlow')}
\newcommand\DAFlowpr{\mathcal{D}(\AFlow')}

\newcommand\Stabf{{\mathcal{S}}(f)}

\newcommand\StabIdf{{\mathcal{S}}_{\id}(f)}

\newcommand\EidAV[1]{\mathcal{E}_{\id}(\AFlow,\Vman)^{#1}}
\newcommand\DidAV[1]{\mathcal{D}_{\id}(\AFlow,\Vman)^{#1}}

\newcommand\EAV{\mathcal{E}(\AFlow,\Vman)}
\newcommand\DAV{\mathcal{D}(\AFlow,\Vman)}

\newcommand\zer{\mathbf{0}}

\newcommand\Wtop{\mathsf{W}}
\newcommand\Stop{\mathsf{S}}
\newcommand\Wr[1]{\Wtop^{#1}}
\newcommand\Sr[1]{\Stop^{#1}}

\newcommand\Cont[1]{\mathcal{C}^{#1}}
\newcommand\Cr[3]{\Cont{#1}(#2,#3)}
\newcommand\Ci[2]{\Cr{\infty}{#1}{#2}}

\newcommand\contWW[2]{\Wtop^{#1,#2}}
\newcommand\contSS[2]{\Stop^{#1,#2}}

\newcommand\orb{o}

\newcommand\FixF{\Sigma}
\newcommand\FixA{\Sigma_{\AFld}}
\newcommand\FixB{\Sigma_{\BFld}}

\newcommand\GVp{\Gamma^{+}_{\Vman}}

\newcommand\imGVp{\ShAV(\GVp)}

\newcommand\Ka{K_{\afunc}}

\newcommand\fVW{\mathcal{U}_{\Vman}}

\newcommand\prQ{p.b.i.}

\newcommand\typeZero{\mathrm{(Z)}}
\newcommand\typeLinear{\mathrm{(L)}}

\newcommand\typeHamilt{\mathrm{(H)}}
\newcommand\typeHamiltExtr{\mathrm{(HE)}}
\newcommand\typeHamiltNonExtr{\mathrm{(HS)}}

\newcommand\trans{B}

\newcommand\imShA{Sh(\AFlow)}
\newcommand\imShAV{Sh(\AFlow,\Vman)}
\newcommand\imShAVi{Sh(\AFlow,\Vman_{i})}
\newcommand\imShAWV{Sh(\AFlow_{\Wman},\Vman)}
\newcommand\imShBeV{Sh(\BeFlow,\Vman)}
\newcommand\imShBV{Sh(\BFlow,\Vman)}

\newcommand\fixtp[1]{\Sigma^{#1}}
\newcommand\fixtpe[1]{\Sigma^{#1'}}

\newcommand\FixLinE{\fixtpe{\typeLinear}}
\newcommand\FixLinPHE{\fixtpe{\typeLinPH}}
\newcommand\FixLinNilpE{\fixtpe{\typeLinNilp}}

\newcommand\FixHamNonExtrE{\fixtpe{\typeHamiltNonExtr}}

\title{Local inverses of shift maps along orbits of flows}
\address{Topology dept., Institute of Mathematics of NAS of Ukraine, \linebreak
Te\-re\-shchenkivs'ka st. 3, Kyiv, 01601 Ukraine 
}

\author{Sergiy Maksymenko}
\email{maks@imath.kiev.ua}

\begin{document}

\begin{abstract}
Let $\AFlow$ be a smooth flow on a smooth manifold $\Mman$ and $\DAFlow$ be the group of diffeomorphisms of $\Mman$ preserving orbits of $\AFlow$.
We study the homotopy type of the identity components $\DidAFlow{r}$ of $\DAFlow$ with respect to distinct Whithey topologies $\Wr{r}$, $(0\leq r\leq \infty)$.
The main result presents a class of flows $\AFlow$ for which $\DidAFlow{r}$ coincide for all $r$ and are either contractible or homotopy equivalent to the circle.
The group $\DidAFlow{0}$ was studied in the author's paper [``Smooth shifts along trajectories of flows'', Topol. Appl. \textbf{130} (2003) 183-204].
Unfortunately that article contains a gap in estimations of continuity of local inverses of the so-called shift map.
The present paper also repairs these estimations and shows that they hold under additional assumptions on the behavior of regular points of $\AFlow$.
\end{abstract}


\maketitle

\section{Introduction}
Let $\Mman$ be a smooth $(\Cont{\infty})$, connected, $m$-dimensional manifold possibly non-compact and with or without boundary.
Let also $\AFld$ be a smooth vector field on $\Mman$ tangent to $\partial\Mman$ and generating a global flow
$\AFlow:\Mman\times\RRR \to \Mman$.
Denote by $\FixA$ (or simply by $\FixF$) the set of singular points of $\AFld$.

Let $\EAFlow$ be the subset of $\Ci{\Mman}{\Mman}$ consisting of maps $\amap$ such that 
\begin{enumerate}
\item
$\amap(\orb) \subset\orb$ for every orbit $\orb$ of $\AFld$;
 \item 
$\amap$ is a local diffeomorphism at every singular point $z\in\FixA$.
\end{enumerate}

Let also $\DiffM$ be the group of $C^{\infty}$-diffeomorphisms of $\Mman$ and
$$\DAFlow \defeq \EAFlow\cap\DiffM$$
be the group of \myemph{diffeomorphisms preserving orbits of $\AFld$}.

For every $r=0,1,\ldots,\infty$ denote by $\EidAFlow{r}$ (resp. $\DidAFlow{r}$) the path-component of the identity map $\id_{\Mman}$ in $\EAFlow$ (resp. $\DAFlow$) with respect to the weak $\Wr{r}$ Whithey topology, see Definition~\ref{defn:r-homotopy}.
In particular $\EidAFlow{0}$ (resp. $\DidAFlow{0}$) consists of maps $\amap$ which are homotopic to $\id_{\Mman}$ in $\EAFlow$ (resp. $\DAFlow$).

Define the following map $\ShA:\Ci{\Mman}{\RRR}\to\Ci{\Mman}{\Mman}$ by $\ShA(\afunc)(x)=\AFlow(x,\afunc(x))$ for $\afunc\in\Ci{\Mman}{\RRR}$ and $x\in\Mman$.
We will call $\ShA$ the \myemph{shift map} along orbits of $\AFld$ and denote by $\imShA$ its image $\ShA(\Ci{\Vman}{\RRR})$ in $\Ci{\Mman}{\Mman}$.
Then the following inclusions hold true:
$$
\imShA\;\subset\;\EidAFlow{\infty}\;\subset\;\cdots\;\EidAFlow{1}\;\subset\;\EidAFlow{0}.
$$
The idea of replacing the time in a flow map with a function was extensively used e.g.  in~\cite{EHoph:1937, Chacon:JMM:1966, Totoki:MFCKUS:1966, Kowada:JMSJ:1972, Parry:JLMS:1972, Kochergin:IANSSSR:1973} for reparametrizations of measure preserving flows and investigations of their mixing properties, see also~\cite[\S1]{Maks:CEJM:2009}.
Smooth shifts functions were applied in~\cite{Maks:TA:2003, Maks:AGAG:2006, Maks:BSM:2006,Maks:hamv2, Maks:MFAT:2009} for study of homotopical properties of certain infinite dimensional groups of diffeomorphisms and their actions on spaces of smooth maps.
Also in~\cite{Maks:CEJM:2009, Maks:ImSh} some applications to parameter rigidity of vector fields are given.

Suppose $\AFld$ is inearizable at each $z\in\FixA$, see e.g.~\cite{Siegel:1952,Sternberg:AmJM:1957,Venti:JDE:1966,Bryuno:TMMO:1971,KondratevSamovol:MZ:1973}.
In~\cite{Maks:TA:2003} the author in particular \emph{claimed\/} that if in addition $\Mman$ is compact, then $\DidAFlow{0}$ and $\EidAFlow{0}$ are either contractible or homotopy equivalent to the circle $S^1$ when endowed with $\Wr{\infty}$ topologies.
Unfortunately, it turned out that in such a generality this statement fails and it is necessary to put additional restrictions on the behavior of regular orbits of $\AFld$.
In fact it was shown that $\imShA=\EidAFlow{0}$ and the mistakes appeared in estimations of continuity of local inverses of $\ShA$, see~\cite[Defn.~15, Th.~17, Lm.~32]{Maks:TA:2003} and also Remarks~\ref{rem:error_cont_div} and~\ref{rem:error-th17} for detailed discussion.

Furthermore, the above description of $\DidAFlow{0}$ was essentially used in another author's paper~\cite{Maks:AGAG:2006} concerning calculations of the homotopy types of stabilizers and orbits of Morse functions on surfaces.

The aim of the present paper is to repair incorrect formulations and proofs of~\cite{Maks:TA:2003} and using~\cite{Maks:hamv2,Maks:CEJM:2009,Maks:MFAT:2009} extend the classes of vector fields for which the homotopy types of $\EidAFlow{r}$ and $\DidAFlow{r}$ can be described (Theorem~\ref{th:main-result}).
In particular it will be shown that the results of~\cite{Maks:TA:2003} used in~\cite{Maks:AGAG:2006} remain true (Theorem~\ref{th:Stabf}).

\subsection{Structure of the paper.}
For the convenience of the reader and due to the length of the paper we will now briefly describe its contents.

In \S\ref{sect:shift_map} we recall the notion of shift map and review (correct) results obtained in~\cite{Maks:TA:2003}.
Also notice that if a flow $\AFlow$ of a vector field $\AFld$ is not global, then we can find a smooth function $\mu:\Mman\to(0,+\infty)$ such that the flow $\AFlow'$ generated by vector field $\AFld'=\mu\AFld$ is global, e.g.~\cite[Cor.~2]{Hart:Top:1983}.
Moreover, $\AFld$ and $\AFld'$ have the same orbit structure, whence $\EAFlow=\EAFlowpr$ and $\DAFlow=\DAFlowpr$.
Thus it could always be assumed that $\AFlow$ is global.
Nevertheless, in \S\ref{sect:open_im_in_im} and \S\ref{sect:open_im_in_im_fixpt} we will consider restriction $\AFld|_{\Wman}$ of $\AFld$ to open subsets $\Wman \subset\Mman$ and compare shift maps of $\AFld$ and $\AFld|_{\Wman}$.
The latter vector field usually generates non-global flow, therefore in this paper it is chosen to work with local flows from the beginning.
In particular, the domain of shift map $\ShA$ of $\AFld$ changes to certain open and convex subset $\funcA\subset\Ci{\Mman}{\RRR}$.

In \S\ref{sect:main_result} we introduce a certain class of vector fields $\FF(\Mman)$ on $\Mman$ and formulate the principal result of the paper: for every $\AFld\in\FF(\Mman)$ its shift map $\ShA$ is a local homeomorphism of $\funcA$ onto its image $\imShA$ with respect to $\Sr{\infty}$ topologies (Theorem~\ref{th:main-result}).
It follows that $\imShA$ is either contractible or homotopy equivalent to the circle, and so is the subspace of $\imShA$ consisting of diffeomorphisms.
On the other hand, in~\cite{Maks:ImSh} it is considered a class of vector fields $\FF'(\Mman)$ which contains $\FF(\Mman)$ and such that for every $\AFld\in\FF'(\Mman)$ the image of its shift map $\imShA$ coincides with either $\EidAFlow{1}$ or $\EidAFlow{0}$ for each $\AFld\in\FF(\Mman)$.
Thus we obtain that $\EidAFlow{r}$ ($\DidAFlow{r}$) coincide with each other at least for all $r\geq1$ and these spaces are either contractible or have homotopy type of $S^1$.
As an application we prove Theorem~\ref{th:Stabf} which extends \cite[Th.~1.3]{Maks:AGAG:2006}.
The rest of the paper is devoted to the proof of Theorem~\ref{th:main-result}.

The assumptions of Theorem~\ref{th:main-result} imply that $\ShA$ is locally injective, therefore in order to prove this theorem it suffices to show that $\ShA$ is open with respect to $\Sr{\infty}$ topologies.
In \S\ref{sect:openness_ShA} we present a characterization of $\contSS{r}{s}$-openness of $\ShA$ for some $r,s\geq0$, (Theorem~\ref{th:sh-open-12}).
First we describe local inverses of $\ShAV$ as certain \myemph{crossed homomorphisms} and then show that like for homomorphisms of groups $\contSS{r}{s}$-openness of the whole map $\ShA$ is equivalent to $\contSS{s}{r}$-continuity of its local inverse defined \myemph{only} on arbitrary small $\Sr{s}$-neighbourhood of the identity map $\id_{\Mman}$.

Further in \S\ref{sect:examples-nontriv-hol} we recall notorious examples of irrational flows on the torus and some modifications of them.
It is shown that shift maps of these flows are not local homeomorphisms onto their images.
This provides counterexamples to~\cite[Th.~1]{Maks:TA:2003} and illustrates certain properties prohibited by Theorem~\ref{th:main-result}.

In \S\ref{sect:regul-ext} we repair~\cite[Lm.~28]{Maks:TA:2003} by giving sufficient conditions for $\contSS{r}{s}$-openness of $\ShA$ to be inherited by regular extensions of vector fields.

\S\ref{sect:LHVectFields} summarizes the formulas for local inverses of shift maps of linear flows obtained in~\cite{Maks:TA:2003} and ``reduced'' Hamiltonian flows of homogeneous polynomials in two variables obtained in~\cite{Maks:CEJM:2009,Maks:hamv2}.
Estimations of continuity of these formulas for linear flows were based on incorrect ``division lemma'' \cite[Lm.~32]{Maks:TA:2003}, see~\S\ref{rem:error_cont_div}.
We will show how to avoid referring to this lemma.

\S\ref{sect:main-result} provides a correct version of~\cite[Th.~17]{Maks:TA:2003}, see Theorem~\ref{th:Sh-open-map}.
It reduces verification of openness of $\ShA$ to openness of a family $\{\ShAVi\}$ of ``local shift maps'' corresponding to any locally finite cover $\{\Vman_i\}$ of $\Mman$ e.g. by arbitrary small smooth closed disks.
The essentially new additional object here is a finite subset $\Lambda'\subset\Lambda$ which appears due to the construction~\eqref{equ:MN_inters} of $\Sr{\infty}$-open set $\ANbh$ and by existence of singularities for which ``local shift map'' is $\contSS{\infty}{\infty}$-open but not $\contSS{r}{d(r)}$-open for some function $d:\NNN\to\NNN$.
Without finiteness of $\Lambda'$ the shift map $\ShA$ is open if its image is endowed with the so-called \myemph{very-strong topology}, \cite{Illman:OJM:2003}.
This effect appears of course only on non-compact manifolds.

\S\ref{sect:repl_M_by_W} splits the ``global analytical problem'' of verification of openness of local shift map $\ShAV$ into the following two problems
\begin{enumerate}
\item[\rm(i)]
openness of local shift map $\ShAWV$ corresponding to the restriction $\AFld|_{\Wman}$ to arbitrary open neighbourhood $\Wman$ of $\Vman$, and 
\item[\rm(ii)]
openness of the image $\imShAWV$ in $\imShAV$.
\end{enumerate}
Due to Theorem~\ref{th:Sh-open-map} the set $\Vman$ can also be taken arbitrary small.
Hence to solve (i) it suffices to consider vector fields in $\RRR^n$.
Thus (i) is a ``local analytical'' problem.
Its solutions for certain vector fields are given in \S\ref{sect:LHVectFields}.

\S\ref{sect:open_im_in_im} and \S\ref{sect:open_im_in_im_fixpt} present sufficient conditions for resolving (ii) at regular and singular points respectively (Theorem~\ref{th:ShAVopen_reg} and Lemma~\ref{lm:prQ_singpt}).
In both cases the assumptions on $\AFld$ are formulated in the terms of dynamical systems theory, and so the problem (ii) can be regarded as a ``global topological'' one.
In the regular case these conditions are also necessary.
Moreover, in the singular case they are relevant with the notion of an isolated block introduced by C.\;Conley and R.\;Easton in~\cite{ConleyEaston:TRAMS:1971}, see Lemma~\ref{lm:isol_block}.

Finally in \S\ref{sect:proof_th:main-result} we prove Theorem~\ref{th:main-result}.

\subsection{Preliminaries.}
Put $\Nz=\NNN\cup\{0\}$ and $\Nzi=\NNN\cup\{0,\infty\}$.

Let $\Mman$ be a smooth manifold of dimension $m$.
Glue to $\Mman$ a collar $\partial\Mman\times[0,1)$ by the identity map $\partial\Mman\times0 \to \partial\Mman$ and denote the obtained manifold by $\Mman'$.
If $\partial\Mman=\varnothing$, then $\Mman'=\Mman$.
Evidently, $\Mman'$ has smooth structure in which it is diffeomorphic with the interior $\Int\Mman$.
Moreover $\Mman$ is a closed subset of $\Mman'$ and $\partial\Mman'=\varnothing$.
We will call $\Mman'$ a \myemph{collaring} of $\Mman$, see Figure~\ref{fig:dsubman}.

\begin{definition}\label{def:D-submanif}
A closed subset $\Vman\subset\Mman$ will be called a \myemph{$\Dm$-subma\-nifold} if there exists an $m$-dimensional  submanifold $\Vman'\subset\Mman'$ possibly with boundary and such that $\Vman=\Mman\cap\Vman'$ and the intersection $\partial\Mman\cap\partial\Vman'$ is transversal, see Figure~\ref{fig:dsubman}.

Denote by $\Int\Vman$ the interior of $\Vman$ in $\Mman$.
Then $\Int\Vman=\Vman\setminus\partial\Vman'$.
We will also say that $\Vman$ is \myemph{$\Dm$-neighbourhood} of every point $z\in\Int\Vman$.
\end{definition}

Let $\Vman\subset\Mman$ be a $\Dm$-submanifold.
If $\Vman\cap\partial\Mman=\varnothing$, then $\Vman$ is a manifold with boundary, otherwise, $\Vman$ is a \myemph{manifold with corners} with $\partial\Mman\cap\partial\Vman'$ as the set of corners.
Evidently, $\Mman$ is a $\Dm$-submanifold of itself, and each $z\in\Int\Mman$ has arbitrary small $\Dm$-neighbourhoods each diffeomorphic to a closed $m$-disk.

Let $\Nman$ be another smooth manifold, $\Nman'$ be its collaring, and $\amap:\Vman\to\Nman$ be a map.
We say that $\amap$ \myemph{is of class $\Cont{r}$}, $r\in\Nzi$, (\myemph{embedding}, \myemph{immersion}, etc.) if it extends to a $\Cont{r}$ map (embedding, immersion, etc.) $U\to\Nman'$ from some open neighbourhood $U$ of $\Vman$ in $\Mman'$ into the collaring $\Nman'$ of $\Nman$.

Then it is known well-known that $\amap$ is $\Cont{r}$ if, and only if, the restriction $\amap|_{\Int\Vman\setminus\partial\Mman}$ is $\Cont{r}$ and all its derivatives have continuous limits when $x$ tends to some point $y\in\partial\Vman'$, see e.g.~\cite{Whitney:TAMS:1934, Seeley:PAMS:1964, Mityagin:UMN:1961, Frerick:JRAM:2007} for manifolds with boundary and e.g~\cite[Pr.~2.1.10]{MargalefOuterelo:1992} for manifolds with corners.

\begin{center}
\begin{figure}[ht]
\includegraphics[height=2cm]{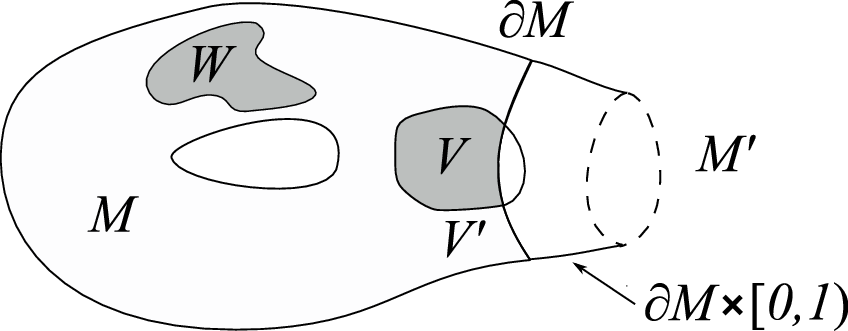} 
\caption{$\Dm$-submanifolds $V$ and $W$}
\protect\label{fig:dsubman}
\end{figure}
\end{center}

Denote by $\Cont{r}(\Vman,\Nman)$, $r\in\Nzi$, the space of all $\Cont{r}$ maps $\Vman\to\Nman$.
Then similarly to~\cite{Hirsch:DiffTop} this space can be endowed with \myemph{weak\/} $\Wr{r}$ and \myemph{strong} $\Sr{r}$ topologies.
If $\Vman$ is compact then $\Wr{r}$ and $\Sr{r}$ coincide for all $r\in\Nzi$.
For a subset $\XX\subset \Cont{r}(\Vman,\Nman)$ denote by $\bigl(\XX\bigr)^{r}_{W}$ (resp. $\bigl(\XX\bigr)^{r}_{S}$) the set $\XX$ with the induced $\Wr{r}$ (resp. $\Sr{r}$) topology, $r\in\Nzi$.

\begin{definition}\label{defn:r-homotopy}
A homotopy $H:\Vman\times I\to\Nman$ will be called an \myemph{$r$-homotopy}, $(r\in\Nzi)$, if for every $t\in I$ the map $H_t=H(\cdot,t):\Vman\to\Nman$ is $\Cont{r}$ and all its partial derivatives in $x\in\Vman$ up to order $r$ are continuous in $(x,t)$.
In other words, $H$ is an \myemph{$r$-homotopy} if, and only if, it yields a continuous path $I\to \Cont{r}(\Vman,\Nman)$ from the standard topology of $I$ to the $\Wr{r}$ topology of $\Cont{r}(\Vman,\Nman)$.
\end{definition}
Hence for every $\amap\in\XX$ its path connected component in $\bigl(\XX\bigr)^{r}_{W}$ consists of all $\bmap\in\XX$ which are $r$-homotopic to $\amap$ in $\XX$.

Now let $\Vman_1, \Vman_2$ be $\Dm$-submanifolds of some smooth manifolds, $\Nman_1,\Nman_2$ be two smooth manifolds, $\XX\subset \Ci{\Vman_1}{\Nman_1}$ and $\YY\subset \Ci{\Vman_2}{\Nman_2}$ be two subsets, and $F:\XX\to\YY$ be a map.

\begin{definition}\label{defn:cont}
We say that $F$ is \myemph{$\contWW{r}{s}$-continuous\/} (\myemph{-open}, etc.) for some $r,s\in\Nzi$ if it is continuous (open, etc.) as a map $F:\bigl(\XX\bigr)^{r}_{W}\to \bigl(\YY\bigr)^{s}_{W}.$
\end{definition}

Similarly we can define \myemph{$\contSS{r}{s}$-continuous\/} (\myemph{-open}, etc.) maps with respect to strong topologies.

Notice that the statement that $F$ is \myemph{$\contWW{\infty}{\infty}$-continuous at some $x\in \XX$} means that for every $s\geq0$ and a $\Wr{s}$-open neighbourhood $V_{F(x)} \subset \YY$ of $F(x)$ there exist $r\geq0$ and a $\Wr{r}$-neighbourhood $U_x \subset \XX$ of $x$ both depending on $s$ and $V_{F(x)}$ such that $F(U_x) \subset V_{F(x)}$.
But in general such a map can be not $\contWW{r}{s}$-continuous for any $r,s$, see e.g.~\cite[p.~93. Eq.~(2)]{Mather_1:AnnMath:1968}, and~\cite{MostowShnider:TrAMS:1985}.

\section{Shift map}\label{sect:shift_map}
Let $\AFld$ be a vector field on $\Mman$ tangent to $\partial\Mman$.
Then for every $x\in\Mman$ its \myemph{orbit} with respect to $\AFld$ is a unique mapping $\orb_x: \RRR\supset(a_x,b_x) \to \Mman$ such that $\orb_x(0)=x$ and $\frac{d}{dt}\orb_x = \AFld(\orb_x)$, where $(a_x,b_x) \subset\RRR$ is the maximal interval on which a map with the previous two properties can be defined.
By standard theorems in ODE the following subset of $\Mman\times\RRR$ 
$$
\domA = \mathop\cup\limits_{x\in\Mman} x \times (a_x, b_x),
$$
is an open, connected neighbourhood of $\Mman\times0$ in $\Mman\times\RRR$.
Then the \myemph{local flow\/} of $\AFld$ is the following map
$$\AFlow: \Mman\times\RRR \; \supset \;\domA\longrightarrow\Mman,
\qquad
\AFlow(x,t) = \orb_x(t).
$$
If $\Mman$ is compact, then $\domA=\Mman\times\RRR$, e.g.~\cite{PalisdeMelo:1982}.

Notice that is $x\in\Mman$ is either fixed or periodic for $\AFlow$, then $x\times\RRR\subset\domA$.

\begin{example}\rm
Let $\AFld(x)=x^2\frac{d}{dx}$ be a vector field on $\RRR^1$.
Then it is easy to see that it generates the following local flow $\AFlow(x,t) = \frac{x}{1-xt}$.
Hence $\AFlow$ is defined on the subset $\domA \subset\RRR^2$ bounded by the hyperbola $xt=1$, see Figure~\ref{fig:domflow}.
\begin{center}
\begin{figure}[ht]
\includegraphics[width=3cm]{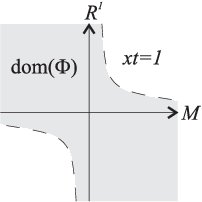} 
\caption{Domain $\domA$ of the flow $\AFlow$ of the vector field $\AFld(x)=x^2\frac{d}{dx}$}
\protect\label{fig:domflow}
\end{figure}
\end{center}
\end{example}
Let $\Vman \subset \Mman$ be either an open subset or a $\Dm$-submanifold and $i_{\Vman}:\Vman\subset\Mman$ be the identity inclusion.
Denote by $\EAV$ the subset of $\Ci{\Vman}{\Mman}$ consisting of maps $\amap:\Vman\to\Mman$ such that (i)~$\amap(\orb\cap\Vman)\subset\orb$ for every orbit $\orb$ of $\AFld$, and (ii)~$\amap$ is an embedding at every singular point $z\in\FixA$.
Let also $\DAV \subset \EAV$ be the subset consisting of immersions $\Vman\to\Mman$.

Let $\EidAV{r}$ and $\DidAV{r}$ be the path-components of $i_{\Vman}$ in the spaces $\EAV$ and $\DAV$ respectively endowed with the corresponding topologies $\Wr{r}$.

Denote by $\funcAV$ the subset of $C^{\infty}(\Vman,\RRR)$ consisting of functions $\afunc$ whose graph $\Gamma_{\afunc}=\{(x,\afunc(x)) \ : \ x\in\Vman\}$ is contained in $\domA$.
Then we can define the following map
$$\ShAV: C^{\infty}(\Vman,\RRR) \;\; \supset  \;\; \funcAV \;\longrightarrow\; C^{\infty}(\Vman,\Mman)$$ 
by $\ShAV(\afunc)(x) = \AFlow(x,\afunc(x))$.
We will call $\ShAV$ the \myemph{(local) shift map} of $\AFld$ on $\Vman$ and denote its image in $\Ci{\Mman}{\Mman}$ by $\imShAV$.
Put 
$$ \GVp=\{\afunc\in\funcAV \ : \ d\afunc(\AFld)(x)>-1 \ \forall x\in\Vman \},$$
where $d\afunc(\AFld)$ is a Lie derivative of $\afunc$ along $\AFld$.
Since $\funcAV$ is an $\Sr{0}$-open and convex subset of $\Ci{\Vman}{\RRR}$ and the map 
$$\LL_{\AFld}:\Ci{\Vman}{\RRR}\to\Ci{\Vman}{\RRR},
\qquad
\LL_{\AFld}(\afunc)=d\afunc(\AFld)
$$
is evidently linear and $\contSS{1}{0}$ continuous, we see that $\GVp$ is convex and $\Sr{1}$-open in $\Ci{\Vman}{\RRR}$.
It also follows from~\cite[Lm.~20 \& Cor.~21]{Maks:TA:2003} that 
\begin{equation}\label{equ:GV}
\GVp=\ShAV^{-1}\bigl(\DidAV{1}\bigr), 
\qquad
\imShAV \subset \EAV.
\end{equation}

\begin{lemma}\label{lm:inclusions}{\rm\cite{Maks:MFAT:2009}}
The following inclusions hold true:
\begin{gather*}
\imGVp \subset \DidAV{\infty} \subset \cdots \subset \DidAV{r} \subset \cdots \subset \DidAV{0}, \\
\imShAV \subset \EidAV{\infty} \subset \cdots \subset \EidAV{r} \subset \cdots \subset \EidAV{0}.
\end{gather*}
If $\imShAV \subset \EidAV{r}$, then $\imGVp=\DidAV{r}$.
\hfill\hfill\qed
\end{lemma}

The set $\ZidAV \defeq \ShAV^{-1}(i_{\Vman})$ will be called the \myemph{kernel\/} of $\ShAV$.

\begin{lemma}\label{lm:Shift-map-prop}
The following properties of shift map hold true:

{\rm(1)}~$\ShAV$ is $\contWW{r}{r}$- and $\contSS{r}{r}$-continuous for all $r\geq 0$, {\rm\cite[Lm.~2]{Maks:TA:2003}}.

{\rm(2)}~$d\mu(\AFld)=0$ for every $\mu\in\ZidAV$, {\rm\cite[Lm.~7]{Maks:TA:2003}}.

{\rm(3)}~Let $\afunc,\bfunc\in\funcAV$. 
Then, {\rm\cite[Lm.~7]{Maks:TA:2003}}, the following conditions are equivalent:
  \begin{enumerate}
    \item[\rm(a)]
       $\ShAV(\afunc)=\ShAV(\bfunc)$ 
    \item[\rm(b)]
     $\afunc-\bfunc\in\funcAV$ and $\ShAV(\afunc-\bfunc)=i_{\Vman}$, i.e. $\afunc-\bfunc\in\ZidAV$.
  \end{enumerate}

{\rm(4)}~$\ShAV$ is locally injective with respect to each of $\Wr{r}$ or $\Sr{r}$ topologies of $\funcAV$ if and only if $\FixA\cap\Vman$ is nowhere dense in $\Vman$, {\rm\cite[Pr.~14]{Maks:TA:2003}}.

{\rm(5)}~Suppose that $\Vman$ is connected and that $\Vman\cap\FixA$ is nowhere dense in $\Vman$.
Then, {\rm\cite[Th.~12(2)]{Maks:TA:2003}},
\begin{enumerate}
 \item[\rm(a)]
either $\ZidAV=\{0\}$ and thus $\ShAV$ is injective.
This case holds if $\Vman$ contains either a non-periodic point of $\AFlow$ or a fixed point $z\in\Vman\cap\FixA$ such that the tangent linear flow $T_{z}\AFlow_t$ on the tangent space $T_z\Mman$ is the identity;
\item[\rm(b)]
or $\ZidAV=\{n\,\nu\}_{n\in\ZZZ}$ for some smooth $\nu:\Vman\to(0,\infty)$.
In this case $\Vman\times\RRR\subset\domA$, so $\funcAV= \Ci{\Vman}{\RRR}$.
\end{enumerate}
\end{lemma}

\section{Main result}\label{sect:main_result}
Let $\Mman$ be a smooth manifold of dimension $m$.
We will introduce a class $\FF(\Mman)$ of smooth vector fields $\AFld$ on $\Mman$ satisfying certain \myemph{topological} conditions on regular points and certain \myemph{analytical} conditions on singular points.
The main result describes the homotopy types of the identity path components of $\EAFlow$ and $\DAFlow$.
First we give some definitions.

{\bf Recurrent points.}
Let $\AFld$ be a vector field on $\Mman$ and $z\in\Mman$ be a regular point of $\AFld$, i.e. $\AFld(z)\not=0$.
Then $z$ is called \myemph{recurrent} if there exists a sequence $\{t_i\}_{i\in\NNN}$ of real numbers such that $\lim\limits_{i\to\infty}|t_i|=\infty$ and $\lim\limits_{i\to\infty}\AFlow(z_,t_i)=z$.
In particular, every periodic point is recurrent.

{\bf First recurrence map.}
Let $z$ be a periodic point of $\AFld$, $\orb_{z}$ be its orbit, and $\trans\subset\Mman$ be and open $(m-1)$-disk passing through $z$ and transversal to $\orb_{z}$.
Then we can define a germ at $z$ of the so-called \myemph{first return (or Poincar\'e) map}  $R:(\trans,z)\to (\trans,z)$ associating to every $x\in\trans$ the first point $R(x)$ at which the orbit $\orb_{x}$ of $x$ first returns to $\trans$.
It is well-known that $R$ is a diffeomorphism at $z$, e.g.~\cite{PalisdeMelo:1982}.

{\bf Reduced Hamiltonian vector fields.}
Let $g:\RRR^2\to\RRR$ be a homogeneous polynomial of degree $p\geq 2$, so we can write
\begin{equation}\label{equ:homog-poly}
g=L_1^{l_1} \cdots L_{a}^{l_a}  \cdot Q_1^{q_1} \cdots Q_{b}^{q_b},
\end{equation}
where $L_i$ is a non-zero linear function, $Q_j$ is an irreducible over $\RRR$ (definite) quadratic form, $l_i,q_j\geq 1$, $L_i/L_{i'}\not=\mathrm{const}$ for $i\not=i'$, and $Q_j/Q_{j'}\not=\mathrm{const}$ for $j\not=j'$.
Denote
$$
D = L_1^{l_1-1} \cdots L_{a}^{l_a-1}  \cdot Q_1^{q_1-1} \cdots Q_{b}^{q_b-1}.
$$
Then $g=L_1 \cdots L_{a} \cdot Q_1 \cdots Q_{b}\cdot D$ and $D$ is the greatest common divisor of the partial derivatives $g'_{x}$ and $g'_{y}$.
The following vector field on $\RRR^2$:
$$
\AFld(x,y)=-(g'_{y}/D)\,\frac{\partial}{\partial x} + (g'_{x}/D)\, \frac{\partial}{\partial y}
$$ 
will be called the \myemph{reduced Hamiltonian} vector field of $g$.
In particular, if $g$ has no multiple factors, i.e. $l_i=q_j=1$ for all $i,j$, then $D\equiv1$ and $\AFld$ is the usual \myemph{Hamiltonian} vector field of $g$.

Notice that $dg(\AFld)\equiv0$ and the coordinate functions of $\AFld$ are homogeneous polynomials of degree $\deg\AFld=a+2b-1$ being relatively prime in the ring of polynomials $\RRR[x,y]$.

{\bf ``Elementary'' singularities.}
Let $\AFld$ be a vector field on $\RRR^{k}$.
It can be regarded as a map $\AFld:\RRR^{k}\to\RRR^{k}$.
Define the following \myemph{types} $\typeZero$, $\typeLinear$, and $\typeHamilt$ of such vector fields.

{\bf Type $\typeZero$}:~\ ~$\AFld(x)\equiv 0$.
 
{\bf Type $\typeLinear$}:~\ ~$\AFld(x)=\afunc(x)\cdot Ax$, \ where $A$ is a non-zero $(k\times k)$-matrix, and $\afunc:\RRR^n\to(0,+\infty)$ is a $\Cont{\infty}$ strictly positive function.
Let $\lambda_{1},\ldots,\lambda_{k}$ be all the eigen values of $A$ taken with their multiplicities.
We can also assume that $A$ has a \myemph{real} Jordan normal form.
Then we will distinguish the following particular cases.

\newcommand\typeLinPH{{\mathrm{(L1)}}}
\newcommand\typeLinNilp{{\mathrm{(L2)}}}
\newcommand\typeLinRotExt{{\mathrm{(L3)}}}
\newcommand\typeLinRot{{\mathrm{(L4)}}}
\newcommand\typeLinNonRot{{\mathrm{(L5)}}}

\newcommand\Rab{\begin{smallmatrix} a & b \\ -b & a\end{smallmatrix}}
\newcommand\Rzb{\begin{smallmatrix} 0 & b \\ b & 0 \end{smallmatrix}}
\newcommand\Rzbi[1]{\begin{smallmatrix} 0 & #1 \\ -#1 & 0 \end{smallmatrix}}
\newcommand\Ematr{\begin{smallmatrix} 1 & 0 \\ 0 & 1\end{smallmatrix}}

$\typeLinPH$~\myemph{$\Re(\lambda_j)\not=0$ for some $j=1,\ldots,k$;}

$\typeLinNilp$~\myemph{The real Jordan normal form has the block}
$$\left.\left\|
\begin{smallmatrix}
0 & 1 &  \\ 
  & 0 & 1  \\
  &   & \cdots & \cdots \\
  &   &     & 0   & 1 \\
  &   &     &     & 0
\end{smallmatrix}
\right\| \right\} n, \qquad n\geq 2;
$$

$\typeLinRotExt$~\myemph{The real Jordan normal form has the block}
$$
\qquad
\left.\left\|
\begin{smallmatrix}
    \Rzb  &    \Ematr   &  \\ 
 \cdots & \cdots & \cdots  \\
           &   \Rzb & \Ematr \\
           &    & \Rzb 
\end{smallmatrix}
\right\| \right\} 2n, \qquad \ (n \geq 2, b\not=0).
$$

In all other cases $A$ is similar to the matrix:
$$\left\|
\begin{smallmatrix}
\Rzbi{b_1}  &         \\ 
            &  \cdots    & \\
            &            & \Rzbi{b_n} \\
            &            & &   0 \\
            &            & &   & \cdots \\
            &            & &   & & 0 
\end{smallmatrix}
\right\|, \qquad n\geq1, \ \ b_j\not=0 \ \forall j=1,\ldots,n.$$
Then we will also distinguish the following two cases:

$\typeLinRot$~
\myemph{there exists $\tau>0$ such that $b_j\tau\in\ZZZ$ for all $j$};

$\typeLinNonRot$~
\myemph{such $\tau$ as in $\typeLinRot$ does not exist}.

{\bf Type $\typeHamilt$}:~\ ~$\AFld(x,y)=\afunc(x,y)\cdot (-g'_{y}/D, g'_{x}/D)$ \ for some strictly positive $\Cont{\infty}$ function $\afunc:\RRR^2\to(0,+\infty)$, and a homogeneous polynomial~\eqref{equ:homog-poly} such that $a+2b-1\geq 2$, so $\AFld$ is not of type $\typeLinear$.
Again we separate the following cases:
\begin{enumerate}
 \item[$\typeHamiltExtr$]
\myemph{$a=0$ and $b\geq2$, i.e. $g$ is a product of at least two distinct definite quadratic forms and has no linear multiples.
In this case $0\in\RRR^2$ is a degenerate global \myemph{extreme} of $g$;}
 \item[$\typeHamiltNonExtr$]
\myemph{$a\geq1$ and $a+2b\geq 3$, so $g$ has linear multiples.
In this case $0\in\RRR^2$ is a degenerate \myemph{saddle} critical point for $g$.}
\end{enumerate}

\begin{remark}\rm
The types $\typeLinear$, $\typeHamilt$, $\typeHamiltExtr$, and $\typeHamiltNonExtr$ coincide with the ones considered in~\cite{Maks:ImSh}.

Also notice that in the case $\typeLinRot$ almost orbits of $\AFld$ are periodic, while in the case $\typeLinNonRot$ almost orbits are non-periodic and recurrent.
\end{remark}

{\bf Regular extensions and products of vector fields.}
Let $\Mman,\Nman$ be manifolds, $\BFld:\Mman\to T\Mman$ be a vector field on $\Mman$, and $\AFld$ be a vector field on $\Mman\times\Nman$.
Then $\AFld$ can be regarded as a map
$\AFld: \Mman\times\Nman \longrightarrow T\Mman\times T\Nman$.
We say that $\AFld$ is a {\em regular extension of $\BFld$} if 
$$\AFld(x,y)=(\BFld(x),H(x,y)), \qquad (x,y)\in\Mman\times\Nman,$$
for some smooth map $H:\Mman\times\Nman\to T\Nman$, so $\BFld$ is the ``first'' coordinate of $\AFld$ and does not depend on $y$, see e.g.~\cite{Blackmore:JDE:1973,Venti:JDE:1966}.
If $H$ does not depend on $x$, i.e. is a vector field on $\Nman$, then 
$$\AFld(x,y)=(\BFld(x),H(y))$$
will be called the \myemph{product} of $\BFld$ and $H$.
Moreover, if $H\equiv0$, i.e.\! $$\AFld(x,y)=(\BFld(y),0),$$ then $\AFld$ is said to be a \myemph{trivial extension} of $\BFld$.

In the case $\Nman=\RRR^{n}$ we will also say that $\AFld$ is a (regular or trivial) \myemph{$n$-extension}.
Evidently, a regular (trivial) $m$-extension of a regular (trivial) $n$-extension is a regular (trivial) $(m+n)$-extension, see e.g.~\cite{Venti:JDE:1966,Maks:TA:2003}.
The following simple statement is left for the reader:
\begin{lemma}
{\rm(a)}~a trivial extension is the same as a product with zero vector field;

{\rm(b)}~a product of vector fields is a regular extension of each of them;

{\rm(c)}~every non-zero linear vector field is a regular extension of a linear vector field defined by one of the following matrices:
$\|a\|$, $\left\|\begin{smallmatrix} a & b \\ -b & a \end{smallmatrix}\right\|$, $(b\not=0)$, or $\left\|\begin{smallmatrix} 0 & 1 \\ 0 & 0 \end{smallmatrix}\right\|$.
\end{lemma}

\begin{definition}
Let $z\in\AFld$ and $\mathrm{(T)}$ denotes one of the types $\typeZero$, $\typeLinear$, $\typeLinPH$, etc., defined above.
We will say the germ of $\AFld$ at $z$ \myemph{is of type $\mathrm{(T)}'$} if the germ of $\AFld$ at $z$ is equivalent to a regular extension of some vector field of type $\mathrm{(T)}$.
By $\fixtp{\mathrm{(T)}}$ and $\fixtpe{\mathrm{(T)}}$ we will denote the set of singular points of $\AFld$ or types $\mathrm{(T)}$ and $\mathrm{(T)}'$ respectively.
\end{definition}

Evidently, a singular point can belong to distinct types.

\begin{definition}\label{defn:classFF}
We will say that a $\Cont{\infty}$ vector field $\AFld$ on a $\Cont{\infty}$ manifold $\Mman$ \myemph{belongs to class $\FF(\Mman)$} if it satisfies the following conditions:
\begin{enumerate}
 \item[\rm(a)]
$\AFld$ is tangent to $\partial\Mman$ and $\FixA$ is nowhere dense in $\Mman$;
 \item[\rm(b)]
every non-periodic regular point of $\AFld$ is non-recurrent;
 \item[\rm(c)]
for every periodic point $z$ of $\AFld$ the germ at $z$ of its first recurrence map $R:(D,z)\to(D,z)$ is either periodic or the tangent map $T_{z}R:T_z D\to T_z D$ has eigen value $\lambda$ such that $|\lambda|\not=1$;
 \item[\rm(d)]
for every $z\in\FixA$ the germ of $\AFld$ at $z$ is either 
\begin{itemize}
 \item[$(*)$] 
a \myemph{product} of finitely many vector fields each of which is of type $\typeLinear$ or $\typeHamilt$, or
 \item[$(**)$]
belongs to one of the types $\typeLinPH'$, $\typeLinNilp'$, $\typeLinRotExt'$, or $\typeHamilt'$, but $z$ is an \myemph{isolated} singular point of $\AFld$.
\end{itemize}
\end{enumerate}
\end{definition}
Thus if the germ of $\AFld$ at $z$ is a regular extension of an $\typeLinRot$-vector field, then $\AFld$ must in fact be a product of such a vector field with vector fields of types $\typeLinear$ or $\typeHamilt$ only.
Also notice that singularities of type $\typeLinNonRot$ are not allowed at all since they have recurrent orbits.
Moreover, we will present examples of vector fields with singularities of type $\typeLinNonRot$ for which the statement of our main result fails, see the end of\;\S\ref{sect:examples-nontriv-hol}.

The following theorem summarizes the results obtained in~\cite{Maks:ImSh} and in this paper.

\begin{theorem}\label{th:main-result}{\rm c.f.~\cite[Th.~1]{Maks:TA:2003} and~\cite{Maks:ImSh}.}
Let $\AFld\in\FF(\Mman)$, $\ShA$ be the shift map of $\AFld$, and
$\Gamma^{+}=\{\afunc\in \funcA : d\afunc(\AFld)(x)>-1 \ \forall x\in\Mman \}$.
Then 
$$\imShA=\EidAFlow{1}, \qquad \ShA(\Gamma^{+})=\DidAFlow{1}.$$
If $\FixF=\FixLinE\cup\FixHamNonExtrE$, then  $\imShA=\EidAFlow{0}$ and $\ShA(\Gamma^{+})=\DidAFlow{0}$.

Suppose that $\overline{\FixF\setminus(\FixLinPHE\cup\FixLinNilpE)}$ is compact (this is always true for compact $\Mman$).
Then both maps 
$$
\ShA:\funcA \to \EidAFlow{1},
\qquad \qquad 
\ShA|_{\Gamma}:\Gamma\to \DidAFlow{1}
$$
are either homeomorphisms or $\ZZZ$-covering maps with respect to $\Sr{\infty}$ topologies.
If $\Mman$ is compact, then the inclusion $\DidAFlow{1}\subset\EidAFlow{1}$ is a homotopy equivalence and both spaces are either contractible or homotopy equivalent to the circle.
If $\AFld$ has at least one non-closed orbit, or a singular point at which the linear part (i.e. $1$-jet) of $\AFld$ vanishes, then $\DidAFlow{1}$ and $\EidAFlow{1}$ are contractible.
\end{theorem}

The first statement about the image of shift maps is established in~\cite{Maks:ImSh} under more general assumptions on $\AFld$.
Therefore we will be proving the second part of Theorem~\ref{th:main-result}, see \S\ref{sect:proof_th:main-result}.
Notice that in comparison with~\cite[Th.~1]{Maks:TA:2003} two additional assumptions are added: ``non-ergodicity'' condition (like absence of recurrent orbits) and compactness of certain subset of singular points.
On the other hand the classes of admissible singularities for $\AFld$ are extended due to results of\;\cite{Maks:hamv2, Maks:CEJM:2009, Maks:reparam-sh-map}.

\subsection{Functions on surfaces.}\label{sect:func-on-surf}
As an application of Theorem~\ref{th:main-result} we will now show that~\cite[Th.~1.3]{Maks:AGAG:2006} which used incorrect results of~\cite{Maks:TA:2003} remains true.
Moreover, we extend the latter theorem to a large class of functions on surfaces with degenerate homogeneous singularities satisfying certain ``secondary'' non-degeneracy conditions.

Let $M$ be a compact surface, $P$ be either the real line $\RRR$ or the circle $S^1$.
For a smooth map $f:M\to P$ denote by $\Sigma_{f}$ the set of its critical points.
Let also $\Stabf=\{h\in \DiffM : f\circ h=f \}$ be the stabilizer of $f$ with respect to the action of $\DiffM$ on $\Ci{M}{P}$ and be the identity component of $\Stabf$ with respect to the $\Wr{\infty}$ topology.

\begin{theorem}\label{th:Stabf}{\rm c.f. \cite[Th.~1.3]{Maks:AGAG:2006}.}
Let $f\in\Ci{M}{P}$.
Suppose that 

{\rm(i)}~$f$ takes constant values on connected components of $\partial M$ and has no critical points on $\partial M$; 

{\rm(ii)}~for every $z\in\Sigma_{f}$ the map $f$ is $\Cont{\infty}$ equivalent near $z$ to a homogeneous polynomial $g_{z}$ \myemph{without multiple factors} such that $\deg g_z\geq2$.

If $\Mman$ is orientable and $\Sigma_{f}$ consists of non-degenerate local extremes only, i.e.\! $f$ is a Morse map without critical points of index $1$, then $\StabIdf$ is homotopy equivalent to the circle.
In all other cases $\StabIdf$ is contractible as well.
\end{theorem}
\begin{proof}[Sketch of proof.]
It suffices to assume that $M$ is orientable.
A non-orientable case will follow from orientable one by arguments of~\cite[\S4.7]{Maks:AGAG:2006}.

Similarly to~\cite[Lm.~5.1]{Maks:AGAG:2006} using (i) and (ii) it is possible to construct a vector field $\AFld$ on $M$ with the following properties:
\begin{enumerate}
\item[\rm(A)] $df(\AFld)=0$, in particular $\AFld$ is tangent to $\partial M$;
\item[\rm(B)] $\AFld(z)=0$ iff $z\in M$ is a critical point of $\AFld$, i.e. $\FixA=\Sigma_{f}$;
\item[\rm(C)] for every critical point $z\in\Sigma_{f}$ there exists a local presentation $f:\RRR^2\to\RRR$ of $f$ in which $z=0$, $f$ is a homogeneous polynomial of degree $\geq2$ and without multiple factors, and $\AFld(x,y)=(-f'_{y},f'_x)$ is a Hamiltonian vector field of $f$.
\end{enumerate}
Then it follows from (A), (B), and arguments of \cite[Lm.~3.5]{Maks:AGAG:2006} that $\StabIdf=\DidAFlow{\infty}$.

Notice that $\AFld$ belongs to class $\FF(M)$. 
Indeed, conditions (a) and (b) are evident. 
Moreover, for each periodic point of $\AFld$ its first recurrent map is the identity, which implies (c).

Finally by (C) each non-degenerate saddle of $f$ is of type $\typeLinPH$, 
each non-degenerate local extreme of $f$ is of type $\typeLinRot$, while all degenerate critical points of $f$ are of types $\typeHamilt$.
This implies (d).

Hence $\StabIdf=\DidAFlow{\infty}$ is either contractible or homotopy equivalent to $S^1$.
If $\Mman$ is orientable and $\Sigma_{f}$ consists of non-degenerate local extremes only, then $f$ is Morse and belongs to one of the types (A)-(D) of\;\cite[Th.\;1.9]{Maks:AGAG:2006}.
In this case the corresponding shift map is not injective, whence $\StabIdf$ is homotopy equivalent to $S^1$.

In all other cases $\ShA$ is injective and so $\StabIdf$ is contractible.
Indeed, if $f$ has a saddle critical point then $\AFld$ has a on-closed orbit, while if $f$ has a degenerate local extreme $z$ then $j^1\AFld(z)=0$.
\end{proof}

\section{Openness of $\ShA$}\label{sect:openness_ShA}
Let $\Vman\subset\Mman$ be a $\Dm$-submanifold.
Our aim is to find sufficient conditions for the shift map $\ShAV:\funcAV\to\imShAV$ to be a local homeomorphism with respect to $\Sr{\infty}$ topologies.
Since $\ShAV$ is $\contSS{r}{r}$-continuous for all $r\in\Nzi$, this is equivalent to $\contSS{\infty}{\infty}$-openness of $\ShAV$.
Moreover, as $\ShAV$ is locally injective, we can also require that its local inverses are $\contSS{\infty}{\infty}$-continuous and defined on $\Sr{\infty}$-open subsets of $\imShAV$.

In this section we show that $\contSS{r}{s}$-openness of $\ShAV$ for some $r,s\in\Nzi$ is equivalent to $\contSS{s}{r}$-continuity of its local inverse defined \myemph{only} on some neighbourhood of the identity inclusion $i_{\Vman}:\Vman\subset\Mman$ in $\imShAV$.
\begin{definition}
Let $A$ be a group and $S$ be a semigroup with unit $e$.
Then the \myemph{right action} of $S$ on $A$ is a map $*:A\times S \to A$ such that $\afunc * e = \afunc$ and $\afunc *(\amap\,\bmap)=(\afunc*\amap)*\bmap$ for all $\afunc\in A$ and $\amap,\bmap\in S$.
A map $\SectShA:S\to A$ is called a \myemph{crossed homomorphism} if
$$
\SectShA(\amap\,\bmap) = ( \SectShA(\amap)*\bmap )\, \SectShA(\bmap).
$$
\end{definition}
Suppose for the moment that $\AFld$ generates a global flow $\AFlow$.
Put $S=\imShA$ and $A=\funcA=\Ci{\Mman}{\RRR}$.
Then $A$ is an abelian group and by~\cite[Eq~(8)]{Maks:TA:2003} $S$ is a subsemigroup of $\Ci{\Mman}{\Mman}$ acting from the right on $A$ as follows:
$$
\afunc\circ s:\Mman \xrightarrow{s} \Mman \xrightarrow{\afunc} \RRR,
\qquad \afunc\in A, \, s\in S.
$$
Suppose also that the shift map $\ShA$ of $\AFlow$ is injective, i.e., $\ShA:\Ci{\Mman}{\RRR}\to\imShA$ is a bijection.
Then it follows from~\cite[Eq~(8)]{Maks:TA:2003} that the inverse map 
$$
\SectShA=\ShA^{-1}:\imShA\to\Ci{\Mman}{\RRR}
$$
is a crossed homomorphism, i.e.
$$
\SectShA(\amap\circ\bmap) = \SectShA(\amap)\circ\bmap + \SectShA(\bmap), \qquad \forall \amap,\bmap\in S.
$$
If $\ShA$ is not injective, then the local inverse of $\ShA$ near $\id_{\Mman}$ is a ``\myemph{local crossed homomorphism}''.
Moreover, if $\AFlow$ is not a global flow, then $\funcA$ is an open subset of $\Ci{\Mman}{\RRR}$ containing the zero function, i.e. a ``\myemph{local group}''.

Notice that if $\imShA$ were a group and $\ShA^{-1}$ were a homomorphism, then continuity of $\ShA^{-1}$ would be equivalent to its continuity at the unit element $\id_{\Mman}$ only, which of course is a simpler problem.
In our case $\imShA$ is just a semigroup, $\ShA^{-1}$ is a local crossed homomorphism, and $\funcAV$ is in general a local group.
Nevertheless we will now show that continuity of local inverses of $\ShA$ is equivalent to continuity of the local inverse of $\ShA$ at $i_{\Vman}:\Vman\subset\Mman$.
\begin{theorem}\label{th:sh-open-12}
Let $\Vman\subset\Mman$ be a $\Dm$-submanifold and $r,s\in\Nzi$.
Then the following conditions are equivalent:
\begin{enumerate}
 \item[\rm(1)]
The shift map $\ShAV:\funcAV\to\imShAV$ is $\contSS{r}{s}$-open;
 \item[\rm(2)]
For every $\afunc\in\funcAV$ there exists an $\Sr{s}$-neighbourhood $\Nbh_{\amap}$ of $\amap=\ShAV(\afunc)$ in $\imShAV$ and an $\contSS{s}{r}$-continuous section of $\ShAV$, i.e. a map
$$
\SectShA:\Nbh_{\amap} \; \longrightarrow \; \Ci{\Vman}{\RRR},
$$
such that $\SectShA(\amap) = \afunc$ and $\ShAV \circ \SectShA = \id(\Nbh_{\amap})$.
In other words
$$
\amap(x) = \AFlow(x,\SectShA(\amap)(x)),
$$
for all $\bmap\in \Nbh_{\amap}$ and $x\in\Vman$.
 \item[\rm(3)]
Property {\rm(2)} holds for the zero function $\afunc=\zer:\Vman\to\RRR$ and the identity inclusion $\amap=\ShAV(\zer)=i_{\Vman}:\Vman\subset\Mman$.
\end{enumerate}
\end{theorem}
\proof
(1)$\Rightarrow$(2).
Suppose that $\ShAV:\funcAV\to\imShAV$ is $\contSS{r}{s}$-open.
Since $\ShAV$ is locally injective with respect to the $\Sr{0}$ topology, condition (1) means that for every $\afunc\in\funcAV$ there exists an $\Sr{r}$-open neighbourhood $\Mbh_{\afunc}$ in $\funcAV$ such that 
\begin{enumerate}
\item[\rm a)]
the restriction $\ShAV|_{\Mbh_{\afunc}}:\Mbh_{\afunc}\to \ShAV(\Mbh_{\afunc})$ is a bijection,
\item[\rm b)]
$\ShAV(\Mbh_{\afunc})$ is $\Sr{s}$-open in $\imShAV$, so $\ShAV(\Mbh_{\afunc})\!\!=\!\imShAV\!\cap\!\Nbh_{\amap}$, for some $\Sr{s}$-open $\Nbh_{\amap}$ neighbourhood of $\amap=\ShAV(\afunc)$ in $\Ci{\Vman}{\Mman}$,
\item[\rm c)]
the inverse map $\SectShAV=\ShAV^{-1}:\ShAV(\Mbh_{\afunc})\to\Mbh_{\afunc}$ is $\contSS{s}{r}$-continuous.
\end{enumerate}
Then $\Nbh_{\amap}\defeq\ShAV(\Mbh_{\afunc})=\imShAV\cap\Nbh_{\amap}$ and $\SectShAV$ satisfy condition (2).

The implication (2)$\Rightarrow$(3) is evident.

(3)$\Rightarrow$(1).
It suffices to show that for every $\afunc\in\funcAV$ there exists an $\Sr{r}$-open neighbourhood $\Mbh_{\afunc}$ in $\funcAV$ satisfying conditions a)-c) above.
Condition (3) means that such a neighbourhood $\Mbh_{0}$ exists for the zero function $\zer$.

For every $\afunc\in \Ci{\Vman}{\RRR}$ define the following subset of $\Ci{\Vman}{\Mman}$
$$\Unbh_{\afunc} = \{ \amap\in \Ci{\Vman}{\Mman} \ : \ (\amap(x),\afunc(x))\in\domA, \forall x\in \Vman\}$$
and consider a map $q_{\afunc} : \Unbh_{\afunc} \to \Ci{\Vman}{\Mman}$ defined by
$$
q_{\afunc}(\amap)(x) = \AFlow(\amap(x),\afunc(x)), \qquad \amap\in \Unbh_{\afunc}, \ x\in\Vman.
$$
It is easy to see that $\Unbh_{\afunc}$ is $\Sr{0}$-open in $\Ci{\Vman}{\Mman}$.
Also notice that, $q_{\afunc}$ is $\contSS{r}{r}$-continuous for all $r\in\Nzi$ and that $i_{\Vman}\in\Unbh_{\afunc}$ if and only if $\afunc\in\funcAV$.

\begin{lemma}\label{lm:q_bfunc}
The image of $q_{\afunc}$ coincides with $\Unbh_{-\afunc}$ and $q_{-\afunc}$ is its inverse, i.e., $q_{-\afunc}=q_{\afunc}^{-1}:\Unbh_{-\afunc}\to\Unbh_{\afunc}$.
Thus $q_{\afunc}$ is an $\contSS{r}{r}$-homeomorphism $(\forall r\in\Nzi)$ between the $\Sr{0}$-open sets $\Unbh_{\afunc}$ and $\Unbh_{-\afunc}$.
\end{lemma}
\begin{proof}
Let $\amap\in\Unbh_{\afunc}$, i.e. $(\amap(x),\afunc(x))\in\domA$ for all $x\in\Vman$.
Then 
$$\bigl(q_{\afunc}(\amap)(x), -\afunc(x)\bigr) \;=\; \bigl(\AFlow(\amap(x),\afunc(x)), -\afunc(x)\bigr)\;\in\;\domA,$$
$$q_{-\afunc}\circ q_{\afunc}(\amap)(x) = \AFlow\bigl(\AFlow\bigl(\amap(x),\afunc(x)\bigl),-\afunc(x)\bigl) =
\AFlow\bigl(\amap(x),\afunc(x)-\afunc(x)\bigl)=\amap(x).$$
Hence $q_{\afunc}(\Unbh_{\afunc}) \subset \Unbh_{-\afunc}$ and $q_{-\afunc}\circ q_{\afunc}=\id_{\Unbh_{\afunc}}$.
Interchanging $\afunc$ and $-\afunc$ will give the result.
\end{proof}

Now we turn to the proof of Theorem~\ref{th:sh-open-12}.
For every $\afunc\in\funcAV$ set
$$
\Mbh_{\afunc} \defeq \bigl(\Mbh_{0} \cap \ShAV^{-1}(\Unbh_{\afunc})  \; + \; \afunc \bigr) \; \cap \; \funcAV.
$$
Then $\Mbh_{\afunc}$ is $\Sr{r}$-open in $\funcAV$.

Moreover $\afunc\in\Mbh_{\afunc}$.
Indeed, since $\afunc\in\funcAV$, we have that $i_{\Vman}\in\Unbh_{\afunc}$, whence $\zer\in\ShAV^{-1}(i_{\Vman}) \subset \ShAV^{-1}(\Unbh)$ and thus $\afunc\in\Mbh_{\afunc}$.

We will show that $\Mbh_{\afunc}$ satisfies the conditions a)-c) above.

a) The restriction of $\ShAV$ to $\Mbh_{\afunc}$ is 1-1, since so is $\ShAV|_{\Mbh_{0}}$.

Conditions b) and c) are implied by the following lemma:
\begin{lemma}\label{lm:Sh_Mpr_open}
Let $\Mbh'_{\afunc} \subset \Mbh_{\afunc}$ be any $\Sr{r}$-open subset.
Then $\ShAV( \Mbh'_{\afunc} )$ is $\Sr{s}$-open in $\imShAV$.
\end{lemma}
\begin{proof}
Denote $\Mbh'_{0} =  \Mbh'_{\afunc} -\afunc$.
Then $\Mbh'_{0} \subset \Mbh_{0} \cap \ShAV^{-1}(\Unbh_{\afunc})$ and by the condition b) for $\Mbh_{0}$ there exists an $\Sr{s}$-open subset $\Nbh'$ of $\Ci{\Vman}{\Mman}$ such that
$$
\ShAV(\Mbh'_{0}) = \imShAV \cap \Nbh'.
$$
Define the following $\contSS{r}{r}$-homeomorphism
$$
a_{\afunc}:\Ci{\Vman}{\RRR} \to \Ci{\Vman}{\RRR}, \qquad a_{\afunc}(\bfunc) = \afunc+\bfunc.
$$
Then we obtain the following commutative diagram:
\begin{equation}\label{equ:CD-ShAv_a}
\begin{CD}
\Mbh'_0 @>{\ShAV}>> \ShAV(\Mbh'_0) = \imShAV \cap \Nbh' \cap \Unbh_{\afunc}
\\
@V{a_{\afunc}}VV @VV{q_{\afunc}}V \\
\Mbh'_{\afunc} @>{\ShAV}>> \imShAV \cap q_{\afunc}(\Nbh' \cap \Unbh_{\afunc})
\end{CD}
\end{equation}
simply meaning that $\AFlow\bigl(x,\bfunc(x)+\afunc(x)\bigr) = \AFlow\bigl(\AFlow(x,\bfunc(x)), \afunc(x)\bigr)$ for all $\bfunc\in\Mbh_0$ and $x\in\Vman$.
We claim that $$\ShAV(\Mbh'_{\afunc})\; \;=\; \;\imShAV \;\cap\; q_{\afunc}(\Nbh'\cap \Unbh_{\afunc}).$$
Indeed, by~\eqref{equ:CD-ShAv_a}, $\ShAV(\Mbh'_{\afunc}) \;\subset\; \imShAV \cap q_{\afunc}(\Nbh'\cap \Unbh_{\afunc}).$

Conversely, let $\bmap\in \imShAV \cap q_{\afunc}(\Nbh'\cap \Unbh_{\afunc})$, i.e., there exist $\bfunc\in\funcAV$ and $\amap\in\Nbh'\cap \Unbh_{\afunc}$ such that 
\begin{equation}\label{equ:g_shb_fsha}
\bmap(x) = \AFlow(x,\bfunc(x)) = \AFlow(\amap(x),\afunc(x)), \qquad \forall x\in\Vman.
\end{equation}
Then $\amap(x) = \AFlow(x,\bfunc(x)-\afunc(x))$, i.e. $\amap\in\imShAV$ and therefore
$$\amap  = \ShAV(\bfunc-\afunc) \;\;\in\;\; \imShAV \cap \Nbh' \cap \Unbh_{\afunc} \;=\; \ShAV(\Mbh'_0).$$
Thus there exists $\cfunc\in\Mbh'_0$, possibly distinct from $\bfunc-\afunc$, such that $\amap=\ShAV(\cfunc)$.
Denote $\bfunc' =\cfunc+\afunc$. 
Then $\bfunc'\in\Mbh'_{\afunc}$ and 
$$
\bmap(x)= \AFlow(\amap(x),\afunc(x)) =
\AFlow\bigl(\AFlow(x,\cfunc(x)), \afunc(x)\bigr) = 
\AFlow(x,\bfunc'(x)).
$$
In other words, $\bmap=\ShAV(\bfunc')\in\ShAV(\Mbh'_{\afunc})$.
Lemma~\ref{lm:Sh_Mpr_open} and Theorem~\ref{th:sh-open-12} are completed.
\end{proof}

\section{Examples when shift map is not open}\label{sect:examples-nontriv-hol}
In this section we discuss four examples of well-known flows whose shift maps turn out not to be homeomorphisms onto their images.
They have similar nature, but are given on different types of manifolds.
These example provide counterexamples to~\cite[Th.~1]{Maks:TA:2003}.
All manifolds in this section are compact, therefore we will not distinguish weak and strong topologies.
In all the examples below our vector fields will satisfy the assumptions of the following simple lemma:
\begin{lemma}\label{lm:ShA_not_open}
Let $\AFld$ be a vector field on a compact manifold $\Mman$ tangent to $\partial\Mman$.
Suppose that the shift map $\ShA$ of $\AFld$ is injective and for every $r\in\Nz$, a $\Wr{r}$-neighbourhood $\Nbh$ of $\id_{\Mman}$, and arbitrary large $T>0$ there exists $t\in\RRR$ regarded as a constant function $t:\Mman\to\RRR$ such that $|t|>T$ and $\AFlow_{t}=\ShA(t)\in\Nbh$.
Then $\ShA$ is not $\contWW{r}{s}$-open for any $r,s\in\Nzi$.
In particular, $\ShA$ is not a homeomorphism onto its image with respect to $\Wr{\infty}$ topologies of $\Ci{\Mman}{\RRR}$ and $\imShA$. \qed
\end{lemma}
\begin{proof}
Consider the following $\Wr{0}$-neighbourhood of the zero function:
$$\Mbh_{0}=\{\afunc\in\Ci{\Mman}{\RRR} \ : \  |\afunc(x)|<1 \}.$$
Suppose that there is a $\Wr{r}$-neighbourhood $\Nbh$ of $\id_{\Mman}$ such that $\Nbh\subset\ShA(\Mbh_{0})$.
By assumption there exists a constant function $t>1$ such that $\ShA(t)\in\Nbh$.
Since $\ShA$ is injective, and $t\not\in\Mbh_{0}$, we obtain that $\ShA(t)\not\in\ShA(\Mbh_{0})$, whence $\Nbh\not\subset\ShA(\Mbh_{0})$.
This contradiction implies that $\ShA$ is not $\contWW{r}{s}$-open for any $r,s\in\Nzi$.
\end{proof}

{\bf Irrational flow.} 
For simplicity we will consider the irrational flow on $T^2$.
Let $\mu\in\RRR$ be an irrational number and $\AFlow$ be the \myemph{irrational} flow on the $2$-torus $T^2=\RRR^2/\ZZZ^2$ given by $\AFlow(x,y,t) = (x+t,y+t/\mu)$.
\begin{lemma}\label{lm:irr-flow}
The shift map $\ShA:\Ci{T^2}{\RRR}\to \Ci{T^2}{T^2}$ of $\AFld$ is not $\contWW{r}{s}$-open for any $r,s\in\Nzi$.
\end{lemma}
\begin{proof}
Notice that every orbit of $\AFlow$ is non-closed and everywhere dense, so  $\ShA$ is injective.
First we give convenient formulas for the metrics generating $\Wr{r}$ topologies on $\Ci{T^2}{T^2}$.
Let $\amap:T^2 \to T^2$ be a $C^{\infty}$ map.
Then $\amap$ lifts to some $\ZZZ^2$-equivariant map $\tilde \amap=(\tilde \amap_1, \tilde \amap_2):\RRR^2\to\RRR^2.$
Let $I^2=[0,1]\times[0,1]\subset\RRR^2$ be the fundamental domain for the covering map $p:\RRR^2\to T^2$.
Define the $r$-th norm of $\amap$ by %
{\footnotesize
$$ 
\|f \|^{r} = \sum_{j=1,2}\left(\sup\limits_{(x,y)\;\in\; I^2} 
\min(\{|\tilde f_j|\},1-\{|\tilde f_j|\} )  + \sum_{1 \leq i_1+i_2 \leq r} \;\sup\limits_{(x,y)\;\in\; I^2} \left|\frac{\partial^{i_1+i_2} \tilde f_j}{\partial x^{i_1} \partial y^{i_2}}\right| \right),
$$ 
}%
where $\{|t|\}$ is the fractional part of the absolute value of $t\in\RRR$.

Let $\varepsilon>0$ and
$\Nbh^{r}_{\eps}=\{ \amap\in \Ci{T^2}{T^2}\ : \ \|\amap-\id_{T^2}\|^r<\eps \}$ 
be a base $\Wr{r}$-neighbourhood of $\id_{T^2}=\ShA(\zer)\in \Ci{T^2}{T^2}$ for some $\eps>0$.
We will show that there exists arbitrary large (by absolute value) $n\in\ZZZ$ such that $\ShA(n\mu)=\AFlow_{n\mu}\in\Nbh^{r}_{\eps}$.
Then our lemma will follow from Lemma~\ref{lm:ShA_not_open}.

Notice that $\AFlow_{n\mu}(x,y) =(x+n\mu,y+n) \equiv (x+n\mu,y)$ for all $n\in\ZZZ$, i.e., $\AFlow_{n\mu}$ is just a ``rotation along the first coordinate''.
Since this map is defined by adding constants to coordinates, it follows from formula for $\|\amap\|^{r}$ that for each $r\in\Nz$ the distance between $\id_{T^2}$ and $\AFlow_{n\mu}$ with respect to the $\Wr{r}$ topology is equal to 
$$ \|\AFlow_{n\mu}-\id_{T^2}\|^{r} = \min( \{|n\mu|\}, 1-\{|n\mu|\} )$$ and therefore does not depend on $r$.

Since $\mu$ is irrational, the set $T_{\mu}=\{ \min( \{|n\mu|\}, 1-\{|n\mu|\} )\}_{n\in\ZZZ}$ is everywhere dense in $S^1=\RRR/\ZZZ$.
Hence there are arbitrary large (by absolute value) $n\in\ZZZ$ such that
$ \|\AFlow_{n\mu}-\id_{T^2}\|^{r}< \eps$, i.e.\! $\AFlow_{n\mu}\in\Nbh^{r}_{\eps}$.
\end{proof}

{\bf Irrational flows on a solid torus.}
Let $D^2 =\{z\in\CCC \ : \ |z|\leq 1\}$ be the unit disk in the complex plane, $T=S^1\times D^2$ be the solid torus, and $\mu$ be an irrational number.
Define the following flow on $T$ by $\AFlow(\phi,z,t) = (\phi+t\, \mathrm{mod} 1, e^{2\pi t / \mu}z)$.
Then by the arguments similar to Lemma~\ref{lm:irr-flow} is can be shown that $\AFld$ satisfies assumptions of Lemma~\ref{lm:ShA_not_open}.

The main feature of this example is that $\AFlow$ has periodic orbit $\gamma=S^1\times0$ and all other ones are recurrent.
Let $\trans=1\times D^2$ be and $R:\trans\to\trans$ be the first recurrence map of $\gamma$ defined by $R(z)=e^{2\pi / \mu} z$, i.e. it is the rotation by $\frac{2\pi}{\mu}$.
Then $R$ is not periodic, eigen values of its tangent map $T_{0}R$ at $0\in\trans$ have modulus $1$, and the iterations of $R$ can be arbitrary close to $\id_{\trans}$.
Thus $\AFlow$ satisfies all assumptions but (b) of Definition~\ref{defn:classFF} of class $\FF(T^2)$.

{\bf Periodic linear flows.}
Given $\lambda_1,\ldots,\lambda_n\in\RRR$ define the following linear flow on $\RRR^{2n}=\CCC^{n}$ by
\begin{equation}\label{equ:linear_flow}
\AFlow(z_1,\ldots,z_n,t)=(e^{i\, \lambda_1 t} z_1, \ldots,  e^{i\, \lambda_n t} z_n).
\end{equation}
Evidently, the closed $2n$-disk $D_r$ of radius $r$ and centered at the origin is invariant with respect to $\AFlow$.

\begin{lemma}\label{lm:periodic_flows}
The following conditions are equivalent:
\begin{enumerate}
\item[\rm(1)]
$\AFlow_{\tau}=\id_{\CCC^n}$ for some $\tau>0$, i.e. $\lambda_j\tau\in\ZZZ$ for all $j=1,\ldots,n$;
\item[\rm(2)]
every $z\in\CCC^n$ is either fixed or periodic with respect to $\AFlow$;
\item[\rm(3)]
at least one point $z=(z_1,\ldots,z_n)$ with all non-zero coordinates is periodic.
\end{enumerate}
A flow satisfying one of these conditions will be called \myemph{periodic}.
\end{lemma}
\begin{proof}
The implications (1)$\Rightarrow$(2)$\Rightarrow$(3) are evident.
(3)$\Rightarrow$(1)
Let $z\in \CCC^n$ be a point with all non-zero coordinates and $\theta$ be the period of $z$ with respect to $\AFlow$.
Then $e^{i\, \lambda_j \theta}z_j=z_j\not=0$ for all $j=1,\ldots,n$, whence $\lambda_j\cdot \theta/2\pi \in\ZZZ$.
\end{proof}

\begin{lemma}\label{lm:non-res-flow}
If $\AFlow$ is not periodic then its shift map $\ShA$ is injective and is not $\contWW{r}{s}$-open for any $r,s\in\Nzi$.
\end{lemma}
\begin{proof}
By Lemma~\ref{lm:periodic_flows} every point $z\in\CCC^{n}$ with all non-zero coordinates is non-periodic and it is easy to see that its orbit is dense on some open subset of the sphere $S^{2n-1}_{|x|} = \partial D_{|x|}$ of radius $|x|$.
In particular, this orbit is non-closed, whence the shift map of $\AFlow$ is injective.
Then similarly to Lemma~\ref{lm:irr-flow}, we can find arbitrary large $t\in\RRR$ such that $\AFlow_{t}$ is arbitrary close to $\id_{\RRR^{2n}}$ in any of $\Wr{r}$ topologies.
\end{proof}

{\bf Flow on $S^{2n}$, $n\geq2$.}
We will now extend the last result for the construction of a flow on the sphere $S^{2n}$ with two fixed points at which this flow is linear.

Let $\RRR^{2n}_1$ and $\RRR^{2n}_2$ be two copies of $\RRR^{2n}=\CCC^n$.
Define a diffeomorphism $\eta:\RRR^{2n}_1\setminus\{0\} \to \RRR^{2n}_2\setminus\{0\}$ by $\eta(z) = \frac{z}{\|z\|^2}$, where $\|z\|$ is the usual Euclidean norm in $\RRR^{2n}$.
Then $\eta$ maps every sphere of radius $r$ centered at $0$ to the sphere of radius $1/r$.
Gluing $\RRR^{2n}_1$ and $\RRR^{2n}_2$ via $\eta$ we obtain a $2n$-sphere.

Notice that Eq.~\eqref{equ:linear_flow} defines the flows on $\RRR^{2n}_1$ and $\RRR^{2n}_2$ so that the following diagram is commutative:
$$
\begin{CD}
\RRR^{2n}_1\setminus\{0\} @>{\eta}>> \RRR^{2n}_2\setminus\{0\} \\
@V{\AFlow_t}VV @VV{\AFlow_t}V \\
\RRR^{2n}_1\setminus\{0\} @>{\eta}>> \RRR^{2n}_2\setminus\{0\}.
\end{CD}
$$
Hence these flows determine a unique flow $\AFlow'$ on $S^{2n}$ with two fixed points.
Moreover, $\AFlow'$ is \myemph{linear} on the charts $\RRR^{2n}_1$ and $\RRR^{2n}_2$ at these points.

Now suppose that $\AFlow$ is not periodic.
Then we can find arbitrary large $t\in\RRR$ such that $\AFlow_{t}$ is arbitrary close to $\id_{\RRR^{2n}}$ in any of  $\Wr{r}$ topologies.
This implies that $\AFlow'$ satisfies assumptions of Lemma~\ref{lm:ShA_not_open}.

\section{Regular and trivial extensions}\label{sect:regul-ext}
The results of this section will allow to estimate continuity of shift maps for vector fields of types $\typeLinear$ and $\typeHamilt$.

Let $\Mman,\Nman$ be two manifolds, $\BFld$ be a vector field on $\Mman$, $\AFld$ be some regular extension and $\BeFld$ be the trivial extension of $\BFld$ on $\Mman\times\Nman$.
Thus 
 $$
 \AFld(x,y)=(\BFld(x),H(x,y)),
 \qquad 
 \BeFld(x,y)=(\BFld(x),0), 
 $$
for some smooth map $H:\Mman\times\Nman\to T\Nman$.
By $\AFlow$, $\BeFlow$, and $\BFlow$ we will denote the corresponding local flows, and by $\ShA$, $\ShBe$, and $\ShB$ the corresponding shift maps of $\AFld$, $\BeFld$, and $\BFld$.
It is easy to see that 
\begin{equation}\label{equ:domAdomBedomB}
\domA\;\subset\;\domBe\;=\;\domB\times\Nman.
\end{equation}
Moreover,
\begin{equation}\label{equ:AFlow_BeFlow}
 \AFlow(x,y,t)=(\BFlow(x,t),\mathbf{H}(x,y,t)),
 \qquad
 \BeFlow(x,y,t)=(\BFlow(x,t),y),
\end{equation}
for some smooth map $\mathbf{H}:\domBe\to\Nman$.

Let $\Vman\subset\Mman\times\Nman$ be a connected compact $\Dm$-submanifold.
Then it follows from~\eqref{equ:domAdomBedomB} that $\funcAV$ is $\Wr{0}$-open in $\funcBeV$.
Define the map $$P:\imShAV \to \imShBeV$$ by the rule: if $\amap\in\imShAV$, and $\amap(x,y)=(A(x,y),g(x,y))$, then $P(f)(x,y) = (A(x,y),y)$.
It follows from~\eqref{equ:AFlow_BeFlow} that $P$ is well defined and is $\contWW{r}{r}$-continuous for every $r\in\Nz$.
Moreover, we have the following commutative diagram:
\begin{equation}\label{equ:ShBeAP}
\xymatrix{
\funcAV^{\phantom{A^A}}  \ar[rr]^{\ShAV} \ar@{^{(}->}[d] \ar@{->}[rrd]^{\ShBeV}  & & ~\imShAV \ar[d]^{P} \\
\funcBeV~ \ar[rr]^{\ShBeV} & & ~\imShBeV 
}
\end{equation}

\begin{theorem}\label{th:reg_ext}
Suppose that $\FixB$ is nowhere dense in $\Mman$ and that $\ShBeV$ is $\contWW{r}{s}$-open for some $r,s\in\Nzi$.
\begin{enumerate}
\item[{\rm(1)}]
If $\ShBeV$ is injective, then $\ShAV$ is $\contWW{r}{s}$-open as well.
\item[{\rm(2)}]
If $\ShAV$ is $\contWW{r}{s}$-open, then $P$ is locally injective with respect to the $\Wr{s}$ topology of $\imShAV$.
\end{enumerate}
Suppose in addition that $\ShBeV$ is also $\contWW{s}{t}$-open for some $t\in\Nzi$.
\begin{enumerate}
\item[{\rm(3)}]
If $P$ is locally injective with respect to the $\Wr{s}$ topology of $\imShAV$, then $\ShAV$ is $\contWW{r}{t}$-open.
\item[{\rm(4)}]
If both $\ShBeV$ and $\ShAV$ are not injective, then $P$ is locally injective with respect to the $\Wr{0}$ topology of $\imShAV$, whence, by {\rm(3)}, $\ShAV$ is $\contWW{r}{t}$-open.
\end{enumerate}
\end{theorem}
\begin{proof}
(1) If $\ShBeV:\funcBeV\to\imShBeV$ is a bijection, it follows from~\eqref{equ:ShBeAP} that so is $\ShAV$.
Let $\Mbh \subset\funcAV$ be a $\Wr{r}$-open subset.
Then
$$\ShAV(\Mbh) = P^{-1} \circ \ShBeV(\Mbh)$$
is $\Wr{s}$-open in $\imShAV$ due to $\contWW{r}{s}$-openness of $\ShBeV$ and $\contWW{s}{s}$-continuity of $P$.

(2) Let $\afunc\in\funcAV$ and $\amap=\ShAV(\afunc)$.
We will find a $\Wr{s}$-neigh\-bourhood of $\amap$ such that $P|_{\Nbh}$ is 1-1.

Since both $\ShAV$ and $\ShBeV$ are $\contWW{r}{s}$-open and locally injective with respect to $\Wr{0}$ topologies of $\funcAV$ and $\funcBeV$ respectively, there exists a $\Wr{r}$-neigh\-bourhood $\Mbh$ of $\afunc$ in $\funcAV$ such that $\Nbh=\ShAV(\Mbh)$ is a $\Wr{s}$-neigh\-bourhood of $\amap$ in $\imShAV$ and the restrictions of $\ShAV$ and $\ShBeV$ to $\Mbh$ are 1-1.
Then it follows from~\eqref{equ:ShBeAP} that $P|_{\Nbh}$ is 1-1 as well.

(3) Let $\Mbh \subset \funcAV$ be a $\Wr{r}$-open subset and $\afunc\in\Mbh$.
We will show that there exists a $\Wr{s}$-neighbourhood $\Mbh'\subset\Mbh$ of $\afunc$ such that $\ShAV(\Mbh')$ is $\Wr{t}$-open in $\imShAV$.
Since $\afunc\in\Mbh$ is arbitrary, we will obtain that $\ShAV(\Mbh)$ is $\Wr{t}$-open in $\imShAV$.

It follows from $\contWW{r}{s}$-openness of $\ShBeV|_{\Mbh}$ and $\contWW{s}{s}$-continuity of $P$ that the set $P^{-1}\circ\ShBeV(\Mbh)$ is a $\Wr{s}$-open neighbourhood of $\ShAV(\afunc)$ in $\imShAV$.
Moreover, since $P$ is locally injective with respect to the $\Wr{s}$ topology of $\imShAV$, there exists a $\Wr{s}$-neighbourhood $\Nbh\subset P^{-1}\circ\ShBeV(\Mbh)$ of $\ShAV(\afunc)$ such that the restriction $P|_{\Nbh}:\Nbh\to \imShBeV$ is injective.

Denote $\Mbh'=\Mbh\cap\ShAV^{-1}(\Nbh)$. 
We claim that
$$
\ShAV(\Mbh') = P^{-1}\circ\ShBeV(\Mbh')\cap\Nbh,
$$
whence $\ShAV(\Mbh')$ will be $\Wr{t}$-open in $\imShAV$.
Indeed,
$$
P^{-1}\circ\ShBeV(\Mbh')\cap\Nbh \stackrel{\eqref{equ:ShBeAP}}{=\!=\!=}P^{-1}\circ P \circ \ShAV(\Mbh')\cap\Nbh \stackrel{\text{injectivity of $P|_{\Nbh}$}}{=\!=\!=\!=\!=\!=\!=\!=\!=\!=\!=} \ShAV(\Mbh').
$$

(4) Suppose that both $\ShBeV$ and $\ShAV$ are not injective.
We will show that for some $m>0$ there exists a free $P$-equivariant action of the group $\ZZZ_{m}$ on $\imShAV$ such that $P(\amap)=P(\bmap)$ iff $\amap,\bmap$ belongs to the same $\ZZZ_{m}$-orbit.
This will give us a decomposition $P:\imShAV \xrightarrow{q} \imShAV/\ZZZ_{m} \xrightarrow{\sigma} \imShBeV$ in which $q$ is a covering map, and $\sigma$ is a bijection with the image $P(\imShAV) \subset \imShBeV$.
Then $q$ will be locally injective with respect to the $\Wr{0}$ topology since $\ZZZ_{m}$ is a finite group. 
Hence so will be $P$.

By Lemma~\ref{lm:Shift-map-prop}(5b) $\funcBeV=\funcAV=\Ci{\Vman}{\RRR}$, $\ZidBeV=\{n{\bar\eta}\}_{n\in\ZZZ}$, and $\ZidAV=\{n{\eta}\}_{n\in\ZZZ}$ for some smooth functions ${\bar\eta},{\eta}:\Mman\times\Nman\to(0,\infty)$ such that
$$
\BeFlow(x,y,t+{\bar\eta}(x,y))\equiv\BeFlow(x,y,t),
\qquad
\AFlow(x,y,t+{\eta}(x,y)) \equiv \AFlow(x,y,t).
$$
In other words
\begin{equation}\label{equ:nu_phi_in_ZidBeV}
\begin{array}{ll}
(\BFlow(x,t+{\bar\eta}), y) = (\BFlow(x,t), y), \\ 
(\BFlow(x,t+{\eta}), \mathbf{H}(x,y,t+{\eta})) = (\BFlow(x,t), \mathbf{H}(x,y,t)).
\end{array}
\end{equation}
It follows that
$\BeFlow(x,y,t+{\eta}) = 
(\BFlow(x,t+{\eta}),y) = (\BFlow(x,t),y)= 
\BeFlow(x,y,t),
$
whence ${\eta}\in\ZidBeV$, i.e. ${\eta}=m{\bar\eta}$ for some $m\in\ZZZ$.

In particular, by (2) of Lemma~\ref{lm:Shift-map-prop}, ${\bar\eta}$ is constant along orbits of $\AFld$.
Since for every $\afunc\in\funcAV$ the map $\ShAV(\afunc)$ preserves every orbit of $\AFld$, we have that ${\bar\eta}\circ \ShAV(\afunc) = {\bar\eta}$, whence
$$
\ShAV({\bar\eta}) \circ \ShAV(\afunc) \stackrel{\text{\cite[Eq.~(8)]{Maks:TA:2003}}}{=\!=\!=\!=\!=\!=\!=} \ShAV(\afunc + {\bar\eta}\circ \ShAV(\afunc))  = 
\ShAV(\afunc + {\bar\eta}).
$$
Then we can define an action $*$ of the group $\ZZZ_m$ on $\imShAV$ by
$$
k * \amap = \ShAV(k{\bar\eta}) \circ \amap, \qquad k\in\ZZZ_m,\; \amap\in\imShAV.
$$
Evidently this action is free and by~\eqref{equ:nu_phi_in_ZidBeV} is equivariant with respect to $P$.

Moreover, let $\amap=\ShAV(\afunc)$, $\bmap=\ShAV(\bfunc)$ for some $\afunc,\bfunc\in\funcAV$ and suppose that $P(\amap)=P(\bmap)$, i.e. $\ShBeV(\afunc)=\ShBeV(\bfunc)$.
Then it follows from (3) of Lemma~\ref{lm:Shift-map-prop} that $\afunc-\bfunc=k{\bar\eta}$ for some $k\in\ZZZ$, whence $\amap = \ShAV(k{\bar\eta}) \circ \bmap$.
Moreover, since ${\eta} = m{\bar\eta}$, we may take $k$ modulo $m$.
In other words, $P(\amap)=P(\bmap)$ iff $\amap= k *\bmap$ for some $k\in\ZZZ_{m}$.
\end{proof}

\section{Vector fields of types $\typeLinear$ and $\typeHamilt$}\label{sect:LHVectFields}
The following lemma shows that property of openness of shift maps near singular points is invariant under reparametrizations.
 
\begin{lemma}\label{lm:reparam_shift_map}{\rm c.f.\;\cite{Maks:reparam-sh-map}.}
Let $\AFld$ be a $\Cont{\infty}$ vector field on $\Mman$, $\nu:\Mman\to(0,+\infty)$ a $\Cont{\infty}$ strictly positive function, and $\BFld=\nu\AFld$.
Let also $\Vman\subset\Mman$ be a $\Dm$-submanifold and $\ShAV$ and $\ShBV$ be shift maps for $\AFld$ and $\BFld$ respectively.
Then $\imShAV=\imShBV$.
Moreover, $\ShAV$ is $\contSS{r}{s}$-open iff so is $\ShBV$.
\end{lemma}
\begin{proof}
Define the functions $\afunc:\domA\to\RRR$ and $\bfunc:\domB\to\RRR$ by 
$$
\bfunc(x,s) = \int\limits_{0}^{s}\nu(\BFlow(x,\tau))\,d\tau,
\qquad
\afunc(x,s) = \int\limits_{0}^{s}\frac{d\tau}{\nu(\AFlow(x,\tau))}.
$$
Then it is well-known that for each $\amap\in\funcAV$ and $\bmap\in\funcBV$, see e.g.\cite{Maks:reparam-sh-map},
$$
\BFlow(x,\bmap(x))=\AFlow(x,\bfunc(x,\bmap(x))),
\qquad
\AFlow(x,\amap(x))=\BFlow(x,\afunc(x,\amap(x))),
$$
for all $x\in\Vman$.
Define the following map
$$
\zeta:\funcBV\to\funcAV,
\qquad 
\zeta(\bmap)(x)=\bfunc(x,\bmap(x)).
$$
Then $\zeta$ is a homeomorphism with respect to topologies $\Wr{r}$ for all $r$, and its inverse is given by $\zeta^{-1}(\amap)(x)=\afunc(x,\amap(x))$.
Moreover $\ShBV=\ShAV\circ\zeta$.
Hence $\ShAV$ is $\contSS{r}{s}$-open iff so is $\ShBV$.
\end{proof}

Thus the study of vector fields of types $\typeLinear$ and $\typeHamilt$ is suffices to consider linear and reduced Hamiltonian vector fields only.

\begin{lemma}\label{lm:Ham_vf}{\rm~\cite{Maks:CEJM:2009, Maks:hamv2}.}
Let $g:\RRR^2\to\RRR$ be a homogeneous polynomial, $\BFld$ be its reduced Hamiltonian vector field, $\AFld$ be a trivial $n$-extension of $\BFld$ on $\RRR^n\times\RRR^{2}$, and $\Vman$ be a $\Dm$-neighbourhood of $0\in\RRR^{n+2}$.
Then the shift map $\ShAV$ of $\AFld$ is $\contWW{\infty}{\infty}$-open.
If $\deg\BFld\geq2$, i.e. $\BFld$ is of type $\typeHamilt$, then for every regular $n$-extension of $\BFld$ its shift map is $\contWW{\infty}{\infty}$-open as well.
\end{lemma}
\begin{proof}
It follows from~\cite{Maks:CEJM:2009, Maks:hamv2} that the shift map $\ShBeV$ of $\BeFld$ has a $\contWW{\infty}{\infty}$-continuous local section on some neighbourhood of the identity inclusion $i_{\Vman}:\Vman\subset\RRR^{n+2}$, see also~\cite[Th.~11.1]{Maks:CEJM:2009}. 
By Theorem~\ref{th:sh-open-12} this implies $\contWW{\infty}{\infty}$-openness of $\ShBV$.

Moreover, if $\deg\BFld\geq2$, i.e. the linear part of $\AFld$ at $0$ vanishes, then by Lemma~\ref{lm:Shift-map-prop} $\ShAV$ is injective.
Hence by (1) of Theorem~\ref{th:reg_ext} the shift map of $\BFld$ is $\contWW{\infty}{\infty}$-open as well.
\end{proof}

\begin{lemma}\label{lm:lin_flows}
Let $B$ be a non-zero $(k\times k)$-matrix, $\BFld(y) = By$ be the corresponding linear vector field on $\RRR^{k}$, $\BFlow(y,t) = e^{Bt}y$ be its flow, and $\AFld$ be a vector field on $\RRR^{n+k}$ being a regular (possibly trivial) $n$-extension of $\BFld$ with respect to the origins of $\RRR^{n+k}$ and $\RRR^{k}$ as indicated in the third column of the table below.
Let also $\Vman\subset\RRR^{n+k}$ be a $\Dm$-neighbourhood of the origin $0\in\RRR^{n+k}$.
Then the shift map $\ShAV$ of $\AFld$ is $\contWW{r}{s}$-open for the values $r,s$ described in the following table.

\begin{longtable}{|c|c|p{0.44\textwidth}|c|c|}\hline
& type  & $B$  & $\AFld$  &  openness of $\ShAV$ \\ \hline
$1$ & $\typeLinPH$ & $B = \|a\|$,  $a\not=0$, & regular & $\contWW{r}{r+1}$, $r\geq0$ \\  \hline 
$2$ & $\typeLinPH$ & $B = \left\| \begin{smallmatrix} a & -b \\ b & a  \end{smallmatrix} \right\|$, $a,b\not=0$ & regular & $\contWW{r}{r+2}$, $r\geq0$  \\ \hline 
$3$ & $\typeLinNilp$ & $B = \left\| \begin{smallmatrix} 0 & 1 \\ 0 & 0  \end{smallmatrix} \right\|$,  & regular & $\contWW{r}{r+1}$, $r\geq0$  \\ \hline
$4$ & $\typeLinRotExt$ & $B = \left\| \begin{smallmatrix}
 0  & -b &  1 & 0 \\
  b & 0 &  0 & 1 \\
 0  & 0 &  0 & -b \\
 0  & 0 &  b & 0 
\end{smallmatrix}\right\|$, $b\not=0$ & regular & $\contWW{\infty}{\infty}$ \\ \hline
$5$ & $\typeLinRot$ & $B = \left\| \begin{smallmatrix} 0 & -b \\ b & 0  \end{smallmatrix} \right\|$, $b\not=0$ & trivial  & $\contWW{\infty}{\infty}$  
\\ \hline
$6$ &   \multicolumn{3}{c}{$B = 
\left\|\begin{smallmatrix}
0    & -b_1 &        &      &     \\
b_1 & 0   &        &      &     \\
     &     & \ddots &      &     \\
     &     &        & 0    & -b_l \\
     &     &        & b_l & 0    
\end{smallmatrix}\right\|^{\phantom{A}}_{\phantom{A_1}}$\!\!,  \ $l\geq 2$}  & \\ \cline{1-5}
$6a$ & $\typeLinRot$ & {\rm (a)}~$b_j \tau \in \ZZZ$ for some $\tau>0$ and all $j=1,\ldots,l$ (periodic case) &  trivial & $\contWW{\infty}{\infty}$ \\ \cline{1-5}
$6b$ & $\typeLinNonRot$ & {\rm (b)}~otherwise  (non-periodic case) &  $-$  & $-$ \\ \hline
\end{longtable}
\end{lemma}
\begin{proof}
In the cases 1-3 we will first suppose that $\AFld$ is a trivial $n$-extension of $\BFld$.
Then $\AFld$ is linear and is generated by the following matrix $\left\| \begin{smallmatrix} 0_n & 0 \\ 0 & B  \end{smallmatrix} \right\|$, where $0_n$ is the zero $(n\times n)$-matrix.
Let $\tau$ denote the coordinates in $\RRR^n$.

{\bf Case 1.}
In this case $\AFld(\tau,x)=a x\frac{\partial}{\partial x}$, $(a\not=0)$, $\AFld$ generates the following global flow $\AFlow(\tau,x,t)=(\tau,x e^{at})$ on $\RRR^{n+1}$,  $\ShAV$ is injective (because $\AFld$ has non-closed orbits), its image $\imShAV$ consists of maps $\amap=(\amap_1,\amap_2)\in\Ci{\Vman}{\RRR^{n}\times\RRR}$ satisfying the following conditions:
$$
\amap_1\equiv\tau, \quad
\amap_2(\tau,0)=0, \quad
\frac{\partial \amap_2(\tau,0)}{\partial x}>0,  \quad
x\amap_2(\tau,x)>0 \ (\forall x\not=0),
$$
and the inverse mapping $\ShAV^{-1}:\imShAV \to \Ci{\Vman}{\RRR}$ is given by
\begin{equation}\label{equ:shift_linear_ax}
\ShAV^{-1}(\amap)(\tau,x) = \frac{1}{a} \ln \frac{\amap_2(\tau,x)}{x} =
\frac{1}{a} \ln \int_{0}^{1} \frac{\partial \amap_2(\tau,tx)}{\partial x} dt,
\end{equation}
see~\cite[Eq.~(23) \& (27)]{Maks:TA:2003}. 
Hence $\ShAV^{-1}$ is $\contWW{r+1}{r}$-continuous for all $r\geq0$, whence $\ShAV$ is $\contWW{r}{r+1}$-open.

{\bf Case 2.}
We will regard $\RRR^{n+2}$ as $\RRR^n\times\CCC$.
Since $a\not=0$, we have that $\AFld$ has non-closed orbits, whence again its shift map $\ShAV$ is injective but now the image $\imShAV$ can not be described so simply as in the previous cases.
Let $\afunc\in\Ci{\Vman}{\RRR}$ and $\amap=(\amap_1,\amap_2)=\ShAV(\afunc) \in  \Ci{\Vman}{\RRR^{n}\times\CCC}$.
Then $\amap_1\equiv\tau$.
Notice that we can define complex conjugate $\bar\amap_2$ and its partial derivatives $\frac{\partial \bar \amap_2}{\partial z}$ and $\frac{\partial \bar \amap_2}{\partial \bar z}$ in $z$ and $\bar z$ in a usual way.
Then it follows from~\cite[Eq.~(29) \& Lm.~34]{Maks:TA:2003} that
$$
\afunc(\tau,z) = 
\frac{1}{2a}\ln \frac{\mathrm{Im}(\omega\;\amap_2\; d\bar\amap_2(\bar\omega))}{\mathrm{Im}(\omega\; z \;\bar \omega)} =
\frac{1}{2a}\ln \frac{Im\bigl( \omega \cdot\amap_2 \cdot(\frac{\partial \bar \amap_2}{\partial z} \omega + \frac{\partial \bar \amap_2}{\partial \bar z} \bar\omega) \bigr) }{y (a^2+b^2)}
$$
and that the numerator of the last fraction is equal to zero, when $y=0$.
It follows from this formula and the Hadamard lemma that the expression of $\afunc$ via $\amap$ contains partial derivatives of $\amap$ up to order $2$.
Hence $\ShAV^{-1}$ is $\contWW{r+2}{r}$-continuous for all $r\geq0$.

{\bf Case 3.}
Now $\AFld(\tau,x,y) = y\frac{\partial}{\partial x}$ for $(\tau,x,y)\in\RRR^{n+2}$ and it generates the following flow $\AFlow(\tau,x,y,t) = (\tau,x+yt, y)$.
Then $\ShAV$ is again injective (since $\AFld$ has non-closed orbits), its image $\imShAV$ consists of mappings $\amap=(\amap_1,\amap_2,\amap_3)\in\Ci{\Vman}{\RRR^{n}\times\RRR\times\RRR}$ such that 
$$
\amap_1\equiv\tau,  \qquad
\amap_2(\tau,x,0)=0, \qquad
\amap_3\equiv y,
$$
and the inverse map $\ShAV^{-1}:\imShAV\to \Ci{\Vman}{\RRR}$ is given by
\begin{equation}\label{equ:shift_linear_nilp}
\ShAV^{-1}(\amap) = \frac{\amap_2(\tau,x,y)-x}{y} = 
 \int_{0}^{1} \frac{\partial \amap_2(\tau,x,ty)}{\partial y} dt,
\end{equation}
see~\cite[Eq.~(26)]{Maks:TA:2003}.
Hence $\ShAV^{-1}$ is $\contWW{r+1}{r}$-continuous for all $r\geq0$.

\medskip

Since in the cases 1-3 $\ShAV$ is injective, it follows from (1) of Theorem~\ref{th:reg_ext} that the same estimations of continuity of $\ShAV^{-1}$ hold for regular extensions of $\BFld$.

Notice that in the remaining cases 4-6 $\AFld$ is a regular $n$-extension of the linear vector field 
$\BFld_5=-by\frac{\partial }{\partial x}+ bx \frac{\partial }{\partial y}$ defined by the matrix $B_5=\left\|\begin{smallmatrix}0&-b\\ b&0\end{smallmatrix}\right\|$.
In fact, it is easy to see that this vector field is the Hamiltonian vector field of the homogeneous polynomial  $g(x,y)=\frac{b}{2}(x^2+y^2)$.
Also notice that the shift map $\ShAV$ of any \myemph{trivial} extension $\AFld$ of $\BFld_5$ is not injective and its kernel $\ZidAV$ consists of integer multiples of constant function $\frac{2\pi}{b}$.

{\bf Case 5.}
It follows from Lemma~\ref{lm:Ham_vf} that local inverses of $\ShAV$ are $\contWW{\infty}{\infty}$-continuous.
An independent proof of $\contWW{\infty}{\infty}$-continuouity of $\ShAV$ for this case is also given in\;\cite{Maks:sym-nondeg-topcenter} under more general settings.
Notice that in this case we claim nothing about openness of shift maps of regular extensions $\AFld$ of $\BFld_5$, since due to Theorem~\ref{th:reg_ext} it is necessary to have additional information about $\AFld$.

{\bf Case 4.}
We will regard $\RRR^4$ as $\CCC^2$.
Suppose at first that $\AFld$ is a trivial $n$-extension of $\BFld$.
Since $\BFld$ is a regular $2$-extension of $\BFld_5$, we see that $\AFld$ is a regular $(n+2)$-extension of $\BFld_5$.
Let also $\BeFld$ be a trivial $(n+2)$-extension of $\BFld_5$.
Thus $\AFld$ and $\BeFld$ are defined on $\RRR^n\times\CCC\times\CCC$ and generate the following global flows:
$$
\AFlow(\tau,z_1,z_2,t) =  (\tau, e^{ibt}(z_1+tz_2), e^{ibt}z_2),
\quad
\BeFlow(\tau,z_1,z_2,t) =  (\tau,z_1, e^{ibt}z_2).
$$
Denote $\theta=\frac{2\pi}{b}$ and let  $\Vman\subset\RRR^n\times\CCC^2$ be a $\Dm$-neighbourhood of $0\in\RRR^{n}\times\CCC^2$.
For every $\afunc\in\Ci{\Vman}{\RRR}$ put
$$
\amapB(\afunc)(\tau,z_1,z_2)=e^{ib \,\afunc(\tau,z_1,z_2)}\,z_1,
\qquad 
\bmapB(\afunc)(\tau,z_1,z_2)=e^{ib \,\afunc(\tau,z_1,z_2)}\,z_2,
$$
Then the shift maps $\ShAV$ and $\ShBeV$ of $\AFld$ and $\BFld$ respectively are given by
\begin{equation}\label{equ:ShAV_ShBeV}
\ShAV(\afunc) = \left(\tau,\amapB(\afunc)+ \afunc\cdot \bmapB(\afunc), \; \bmapB(\afunc)\right),
\qquad
\ShBeV(\afunc) = (\tau,z_1, \bmapB(\afunc)).
\end{equation}

Similarly to~\eqref{equ:ShBeAP} define $P:\imShAV\to\imShBeV$ by the following rule: if $\cmap=(\tau,\amap,\bmap)\in\imShAV \subset \Ci{\Vman}{\RRR^{n}\times\CCC\times\CCC}$, then $P(\cmap)(\tau,z_1,z_2)=(\tau,z_1,\bmap(z_1,z_2))$.
Evidently, $\ShBeV=P\circ\ShAV$.

By the case 5 the shift map $\ShBeV:\Ci{\Vman}{\RRR}\to\imShBeV$ is $\contWW{\infty}{\infty}$-open.
We claim that so is $\ShAV:\Ci{\Vman}{\RRR}\to\imShAV$.
Since $\BFld$ has non-closed orbits, $\ShBeV$ is always injective, whence it will follow from (1) of Theorem~\ref{th:reg_ext} that the shift map of any regular extension of $\BFld$ is  $\contWW{\infty}{\infty}$-open as well.

By (3) of Theorem~\ref{th:reg_ext} it suffices to show that $P$ is locally injective with respect to the $\Wr{1}$ topology.
Evidently, $\ShBV^{-1}\circ\ShBV(\afunc) = \{ \afunc + \theta\, n\}_{n\in\ZZZ}$, whence we obtain from~\eqref{equ:ShAV_ShBeV} that
$$
P^{-1} \circ \ShBeV(\afunc) = 
\left\{ \ShAV(\afunc) + \left(0,\theta\,n \cdot \bmapB(\afunc), 0\right)  \ | \ n\in\ZZZ	\right\}.
$$
Now, define the following $C^{1}_{W}$-neighbourhood of $\ShAV(\afunc)$ in $\imShAV$:
$$\Nbh = \{\amap\in\imShAV \ | \ \|\amap - \ShAV(\afunc) \|^1_{V} < |\theta|/4 \}.$$
We claim that the restriction $P|_{\Nbh}$ is 1-1.

Since $\ShAV$ is injective map, it suffices to establish that whenever both $\ShAV(\bfunc)$ and $\ShAV(\bfunc+\theta\,n)$ belong to $\Nbh$ for some $\bfunc\in \Ci{\Vman}{\RRR}$, then $n=0$.
Notice that
$$
\|\bmapB(\bfunc)\|_{\Vman}^{1} \geq \bigl|\bmapB(\bfunc)'_{z_2}\bigr|_{(z_1,z_2)=0} = 
\left| e^{ib\,\bfunc} \, \left(i b\, z_2\, \bfunc'_{z_2} + 1 \right) \right|_{(z_1,z_2)=0} = 1.
$$
Then 
\begin{multline*}
\left\|\ShAV(\bfunc) - \ShAV\left(\bfunc+\theta\,n\right) \right\|^{1}_{\Vman} \leq
\left\|\ShAV(\bfunc) - \ShAV(\afunc) \right\|^{1}_{\Vman} + \\ +
\left\|\ShAV\left(\bfunc+\theta\,n\right) - \ShAV(\afunc) \right\|^{1}_{\Vman} \leq |\theta|/2.
\end{multline*}
On the other hand,
$$
\|\ShAV(\bfunc) - \ShAV(\bfunc+\theta\,n) \|^{1}_{\Vman} = 
\|(0, \theta\,n\, \bmapB(\bfunc),0 ) \|^{1}_{\Vman}  \geq |\theta\,n|.
$$
Hence $n=0$.

{\bf Case 6a.}
In this case $\AFld$ is a regular extension of $\BFld_5$.
Since the flow $\AFlow$ is periodic, we have that the shift maps of $\AFld$ and $\BFld_5$ are not injective and by the case 5 the shift map of $\BFld_5$ is $\contWW{\infty}{\infty}$-open.
Then by (4) of Theorem~\ref{th:reg_ext} so is the shift map of $\AFld$.

{\bf Case 6b.}
In this case $\ShAV$ is injective, but as it is shown in Lemma~\ref{lm:non-res-flow} its inverse map is not even $\contWW{\infty}{\infty}$-continuous.

Lemma~\ref{lm:lin_flows} is completed.
\end{proof}

\begin{remark}\label{rem:error_cont_div} \rm
Incorrect estimations of continuity of local inverses of $\ShAV$ given in~\cite[pages 199-200]{Maks:TA:2003} were bases on at the following ``division lemma'', which was wrongly formulated\footnote{In~\cite[Lm.~32]{Maks:TA:2003} it was claimed that $Z^{-1}$ is $\contWW{r}{r}$-continuous for all $r\in\Nzi$.
But the latter inequality in the proof of~\cite[Lm.~32]{Maks:TA:2003} actually shows $\contWW{r}{r}$-continuity of $Z$ but not of its inverse.} in~\cite{Maks:TA:2003}.
\begin{lemma}\label{lm:Zinv}{\rm c.f.~\cite[Lm.~32]{Maks:TA:2003}.}
Let $\FFF$ be either $\RRR$ or $\CCC$, $\Vman\subset\FFF$ be a $\Dm$-submanifold, and $Z:\Ci{\Vman}{\FFF}\to \Ci{\Vman}{\FFF}$ be a map defined by the formula: $Z(\afunc)(x)=x \cdot \afunc(x)$.
If $\FFF=\RRR$ then the inverse map $Z^{-1}:\im Z \to \Ci{\Vman}{\FFF}$ is $\contWW{r+1}{r}$-continuous for all $r\geq0$. 
If $\FFF=\CCC$, then $Z^{-1}$ is only $\contWW{\infty}{\infty}$-continuous.
\end{lemma}
The case $\FFF=\RRR$ easily follows from the Hadamard lemma, see also~\cite{MostowShnider:TrAMS:1985}.
The case $\FFF=\CCC$ is more complicated and can be established by the methods of~\cite{Maks:CEJM:2009,Maks:hamv2} but during the proof of Lemma~\ref{lm:lin_flows} we avoided referring to it.
\end{remark}

\section{Fragmentation}\label{sect:main-result}
The aim of this section is to repair~\cite[Th.~17]{Maks:TA:2003} by giving a sufficient condition for the shift map $\ShA:\funcA\to\imShA$ to be either a homeomorphism or an infinite cyclic covering map with respect to $\Sr{\infty}$ topologies, see Theorem~\ref{th:Sh-open-map}.

\begin{remark}\label{rem:error-th17}\rm
The error occurred in the third paragraph of~\cite[Th.~17]{Maks:TA:2003}, where it was claimed that 
``\myemph{... the image $\mathcal{N}_i = \varphi_{U_i}(\mathcal{M}'_i)$ is a $C^{r}_{W}$-neighbourhood of $f|_{U_i}$ in $C^{\infty}(U_i,M)$ for all $r\in\mathbb{N}_0$.}''
First of all this phrase contains a misprint: instead of $C^{\infty}(U_i,M)$ the author supposed to be written $\im\varphi_{U_i}$.
But nevertheless the statement that \myemph{$\mathcal{N}_i$ is $C^{r}_{W}$-open in $\im\varphi_{U_i}$} does not follow from the assumptions of~\cite[Th.~17]{Maks:TA:2003} and must be included into the formulation of ``section'' property $(S)^{0}$ of~\cite[Defn.~15]{Maks:TA:2003}.
This is the point which was missed.
\end{remark}
\begin{theorem}\label{th:Sh-open-map}{\rm c.f.~\cite[Th.~17]{Maks:TA:2003}.}
Let $\AFld$ be a vector field on $\Mman$ such that $\FixA$ is nowhere dense.
Suppose that there exists $a\in\Nz$, a function $d:\Nz\to\Nz$, a locally finite cover $\{\Vman_i\}_{i\in\Lambda}$ of $\Mman$ by $\Dm$-submanifolds, and a finite (possibly empty) subset $\Lambda'\subset\Lambda$ such that $\{\Int\Vman_i\}_{i\in\Lambda}$ is also a cover of $\Mman$, and the local shift map 
$$\ShAVi:\funcAVi\to\imShAVi$$
is $\contSS{\infty}{\infty}$-open if $i\in\Lambda'$ and $\contSS{r}{d(r)}$-open for all $r\geq a$ if $i\in\Lambda\setminus\Lambda'$.
Then the shift map $$\ShA:\funcA\to\imShA$$ is $\contSS{\infty}{\infty}$-open.
Moreover, if $\Lambda'=\varnothing$, then $\ShA$ is $\contSS{r}{d(r)}$-open for every $r\geq a$.

Hence if $\ShA:\funcA\to\imShA$ is injective, then it is a homeomorphism with respect to $\Sr{\infty}$ topologies.
Otherwise, $\ShA$ is a $\ZZZ$-covering map. 
\end{theorem}
\begin{proof}
The proof follows the line of \cite[Th.~17]{Maks:TA:2003}.
Since $\FixA$ is nowhere dense, there exists a continuous function $\delta:\Mman\to(0,\infty)$ such that for every periodic point $x$ of period $\theta_x$ we have that $\delta(x)<3\theta_x$, see~\cite[Pr.~14]{Maks:TA:2003}.

Let $\afunc\in\funcA$, $r\geq a$, $d^r$ be a metric on the manifold $J^r(\Mman,\RRR)$ of $r$-jets, $\nu:\Mman\to(0,1)$ be a strictly positive continuous function such that $\nu<\delta$, and 
$$\AMbh \defeq \{\bfunc\in\funcA \ | \ d^r(j^{r}\afunc(x), j^{r}\bfunc(x)) < \nu(x), \forall x\in\Mman\}$$ be a base $\Sr{r}$-neighbourhood of $\afunc$, where $j^{r}\bfunc(x)$ denotes the $r$-jet of $\bfunc$ at $x$.
Then the restriction of $\ShA$ to $\AMbh$ is injective, \cite[Pr.~14]{Maks:TA:2003}.
We will show that $\ShA(\AMbh)$ is $\Sr{\infty}$-open in $\imShA$ and if $\Lambda'=\varnothing$, then $\ShA(\AMbh)$ is even $\Sr{d(r)}$-open in $\imShA$.
This will complete Theorem~\ref{th:Sh-open-map}.

For every $i\in\Lambda$ let
\begin{equation}\label{equ:ABmh_nu_i}
\AMbh_{i} \defeq \{\bfunc\in\funcAVi \ : \ d^r(j^{r}\afunc(x), j^{r}\bfunc(x)) < \nu(x), \forall x\in\Vman_i\}
\end{equation}
be an $\Sr{r}$-neighbourhood of $\afunc|_{\Vman_i}$ in $\funcAVi$. 
It follows from assumption about openness of $\ShAVi$ that
$$
\ShAVi(\AMbh_{i}) = \imShAVi \cap \Nbh_{i}\,,
$$
where $\Nbh_i$ is $\Sr{\infty}$-open in $\Ci{\Vman_i}{\Mman}$ for $i\in\Lambda'$ and $\Sr{d(r)}$-open for $i\in\Lambda\setminus\Lambda'$.
Moreover, the restriction of $\ShAVi$ to $\AMbh_i$ is one-to-one.

Let $p_i:\Ci{\Mman}{\RRR}\to \Ci{\Vman_{i}}{\RRR}$ and $q_i: \Ci{\Mman}{\Mman}\to \Ci{\Vman_{i}}{\Mman}$ be the ``restriction to $\Vman_{i}$'' maps.
Then we have the following commutative diagram
\begin{equation}\label{equ:sh_p__q_sh}
\begin{CD}
 \funcA @>{\ShA}>>  \imShA \\
 @V{p_i}VV    @VV{q_i}V  \\
 \funcAVi       @>{\ShAVi}>>  \imShAVi
\end{CD}
\end{equation}
By definition $\ShA$ and $\ShAVi$ as surjective.
It also follows from definitions of $\funcA$ and $\funcAVi$ and assumption that $\Vman_i$ is a $\Dm$-submanifold, that every $\sigma\in\Ci{\Vman}{\RRR}$ extends to some $\bar\sigma\in\Ci{\Mman}{\RRR}$.
Moreover, if $\sigma\in\funcAVi$, we can assume that $\bar\sigma\in\funcA$, whence  $p_i(\funcA) = \funcAVi$.
Therefore $q_i$ is surjective as well.
It follows from definition that 
$$
\AMbh=\mathop\cap\limits_{i\in\Lambda} p_i^{-1}(\AMbh_i).
$$
Put
\begin{equation}\label{equ:MN_inters}
\ANbh\defeq\mathop\cap\limits_{i\in\Lambda} q_i^{-1}(\ANbh_i).
\end{equation}
Since $\{\Vman_i\}_{i\in\Lambda}$ is a locally finite cover and $\Lambda'$ is \emph{finite}%
\footnote{If $\Lambda'$ were infinite, then $\ANbh$ would be open in the so-called very-strong topology, see~\cite{Illman:OJM:2003}}, it follows from \cite[Lm.~18]{Maks:TA:2003} that $\ANbh$ is $\Sr{\infty}$-open in $\Ci{\Mman}{\Mman}$ and even $\Sr{d(r)}$-open if $\Lambda'=\varnothing$.
We will now show that 
\begin{equation}\label{imShift_N__M}
\ShA(\AMbh) = \imShA \cap \ANbh.
\end{equation}
This will complete our theorem.

It follows from~\eqref{equ:sh_p__q_sh} and~\eqref{equ:MN_inters} that $\ShA(\AMbh) \subset \imShA \cap \ANbh.$
Conversely, let $\bmap\in\imShA\cap\ANbh$.
Then $\bmap|_{\Vman_{i}} = q_i(\bmap) \in \imShAVi\cap\ANbh_i = \ShAVi(\AMbh_i)$ for all $i\in\Lambda$ and $\bmap=\ShA(\bfunc')$ for some $\bfunc'\in \Ci{\Mman}{\RRR}$.
We have to find (possibly another) function $\bfunc\in\AMbh$ such that $\bmap=\ShA(\bfunc)$.
Since the restriction of $\ShAVi$ to $\AMbh_i$ is injective, $\bmap|_{\Vman_{i}}=\ShAVi(\bfunc_i)$ for a \myemph{unique} $\bfunc_i\in\AMbh_i$.

It remains to show that $\bfunc_i=\bfunc_j$ on $\Vman_{i} \cap \Vman_{j}$ for all $i,j\in\Lambda$.
Since $\{\Int\Vman_i\}_{i\in\Lambda}$ is a cover of $\Mman$ as well, the family of functions $\{\bfunc_i\}_{i\in\Lambda}$ will define a unique smooth function $\bfunc\in \mathop\cap\limits_{i\in\Lambda} p_i^{-1}(\AMbh_i) = \AMbh$ such that $\bfunc|_{\Vman_{i}} = \bfunc_i$ and $\ShA(\bfunc)=\bmap$.
This will prove~\eqref{imShift_N__M} and complete our theorem.

Let $x\in \Int\Vman_{i} \cap \Int\Vman_{j}$ for some $i,j\in \Lambda$.
Then $\bmap(x) = \AFlow(x,\bfunc_i(x)) = \AFlow(x,\bfunc_j(x)).$

If $x$ is a non-periodic regular point, then $\bfunc_i(x) = \bfunc_j(x)$.

If $x$ is periodic of period $\theta_x$, then $\bfunc_i(x)-\bfunc_j(x)= b\theta_x$ for some $b\in\ZZZ$.
But $|\bfunc_i(x)-\bfunc_j(x)| \leq |\afunc(x)-\bfunc_i(x)|+ |\afunc(x)-\bfunc_j(x)|<2\delta(x) < \theta_x$, whence $b=0$.

Thus $\bfunc_i=\bfunc_j$ on $(\Int\Vman_{i}\cap  \Int\Vman_{j})\setminus \FixA$.
Since $\FixA$ is nowhere dense, we see that $\bfunc_i=\bfunc_i$ on all of $\Vman_{i}\cap \Vman_{j}$ as well.
\end{proof}

\section{Replacing $\Mman$ with an open subset}\label{sect:repl_M_by_W}
Theorem~\ref{th:Sh-open-map} reduces the verification of openness of the shift map to openness of a family of shift maps $\{\ShAVi\}$, where $\{\Vman_i\}$ is any locally finite cover of $\Mman$ by $\Dm$-submanifolds.
Our next aim is to ``localize'' the verification an openness of $\ShAV$ by replacing all the manifold $\Mman$ with an open neighbourhood $\Wman$ of $\Vman$.

Let $\AFld$ be a vector field on $\Mman$, $\Wman\subset\Mman$ be a connected, open subset, and $\AFlowW:\Wman\times\RRR\supset\domAW\to\Wman$ be the local flow generated by the restriction $\AFld|_{\Wman}$ of $\AFld$ to $\Wman$.
Then 
$$
\domAW \;\subset\; \domA\cap (\Wman\times\RRR)
$$
and $\AFlow=\AFlowW$ on $\domAW$.
Let $\Vman\subset\Wman$ be a $\Dm$-submanifold.
Then $\funcAWV$ is an $\Sr{0}$-open subset of $\funcAV$, and the corresponding shift map $\ShAWV$ of $\AFlowW$ coincides with $\ShAV$ on $\funcAWV$, i.e.
\begin{equation}\label{equ:ShAWV_ShAV}
\ShAWV=\ShAV|_{\funcAWV}\;: \;\funcAWV \;\to\; \imShAWV\;\subset \; \imShAV.
\end{equation}

The following statement characterizes openness of $\ShAV$ via openness of $\ShAWV$.
\begin{theorem}\label{th:charact_sh_open}
Let $\Vman\subset\Wman$ be a $\Dm$-submanifold, and $r,s\in\Nzi$.
Then the following conditions {\rm(A)-(C)} are equivalent:
\begin{enumerate}
 \item[\rm(A)]
The shift map $\ShAV$ is $\contSS{r}{s}$-open;
 \item[\rm(B)]
The shift map $\ShAWV$ is $\contSS{r}{s}$-open and its image $\imShAWV$ is $\Sr{s}$-open in $\imShAV$, i.e. there exists an $\Sr{s}$-open subset $\Nbh\subset\Ci{\Vman}{\Wman}$ such that 
$$\imShAV \cap \Nbh \; = \; \imShAWV.$$
 \item[\rm(C)]
The shift map $\ShAWV$ is $\contSS{r}{s}$-open and there exists an $\Sr{s}$-neigh\-bourhood $\Nbh$ of the identity inclusion $i_{\Vman}:\Vman\subset\Mman$ in $\Ci{\Vman}{\Mman}$ such that $$\imShAV \cap \Nbh \; \subset \; \imShAWV.$$
\end{enumerate}
\end{theorem}
\begin{proof}
(A)$\Rightarrow$(B).
Since $\funcAWV$ is an $\Sr{0}$-open subset of $\funcAV$ and $\ShAV$ is an $\contSS{r}{s}$-open map, it follows that the restriction $$\ShAV|_{\funcAWV}=\ShAWV$$ is an $\contSS{r}{s}$-open map and its image $\imShAWV$ is $\Sr{s}$-open in $\imShAV$.

The implication (B)$\Rightarrow$(C) is evident.

(C)$\Rightarrow$(A).
Suppose that $\ShAWV$ is $\contSS{r}{s}$-open.
Then by (2) of Theorem~\ref{th:sh-open-12} there exists an $\Sr{s}$-neighbourhood $\Unbh$ of the identity inclusion $i_{\Vman}:\Vman\subset\Wman$ in $\Ci{\Vman}{\Wman}$ and an $\contSS{s}{r}$-continuous section of $\ShAWV$ defined on $\Unbh'=\Unbh\cap\imShAWV$:
$$
\SectShA:
\imShAV \supset \imShAWV \supset \Unbh' \;\xrightarrow{~\SectShA~}\; \funcAWV \subset \funcAV,
$$
i.e., $\ShAWV\circ\SectShA=\id_{\Unbh'}$.

Since $\imShAV\cap\Nbh\subset\imShAWV$ and $\Ci{\Vman}{\Wman}$ is an $\Sr{0}$-open subset of $\Ci{\Vman}{\Mman}$, we obtain that $\Unbh''=\Unbh'\cap\Nbh$ is an $\Sr{s}$-neighbourhood of $i_{\Vman}:\Vman\subset\Mman$ in $\Ci{\Vman}{\Mman}$.
Moreover, $\ShAV$ coincides with $\ShAWV$ on $\funcAWV$.
Therefore $\SectShA$ is also a section of $\ShAV$ defined on $\Unbh''$.
Then by (2) of Theorem~\ref{th:sh-open-12} $\ShAV$ is $\contSS{r}{s}$-open.
\end{proof}

Emphasize that in this theorem $\Wman$ is an \myemph{arbitrary} open neighbourhood of $\Vman$.

\section{Openness of $\imShAWV$ in $\imShA$. Regular case.}\label{sect:open_im_in_im}
Let $z$ be a regular point of $\AFld$, i.e. $\AFld(z)\not=0$.
In this section we present necessary and sufficient conditions for $\contWW{r}{r}$-openness of local shift map $\ShAV$, where $\Vman$ is arbitrary small $\Dm$-neighbourhood of $z$, see Theorem~\ref{th:cond_open_imShAWV_reg}.
As a consequence we will obtain the following 
\begin{theorem}\label{th:ShAVopen_reg}
Suppose that $z$ is a regular point of $\AFld$ having one of the following properties:
\begin{enumerate}
 \item[\rm(a)]
$z$ is non-periodic and non-recurrent;
 \item[\rm(b)]
$z$ is periodic and the germ at $z$ of its first recurrence map $R:(\trans,z)\to (\trans,z)$ is periodic;
 \item[\rm(c)]
$z$ is periodic and the tangent map $T_z R: T_z\trans\to T_z\trans$ has at least one eigen value $\lambda$ such that $|\lambda|\not=1$.
\end{enumerate}
Then for any sufficiently small connected $\Dm$-neighbourhood $\Vman$ of $z$ the corresponding shift map $\ShAV$ of $\AFld$ is $\contWW{r}{r}$-open for all $r\geq 1$.
Moreover, if $z$ satisfies either {\rm(a)} or {\rm(b)}, then $\ShAV$ is $\contWW{0}{0}$-open as well.
\end{theorem}
The proof will be given in \S\ref{sect:proof_th:ShAVopen_reg}.

Since $z$ is a regular point of $\AFld$, there exist $\eps>0$, a neighbourhood $\Wman'$ of $z$, and a diffeomorphism
$$
\eta:\Wman'\to\RRR^{n-1}\times(-4\eps,4\eps)
$$
such that in the coordinates $(y,s)$ on $\Wman'$ induced by $\eta$ we have that 
$$
\AFlow((y,s),t)=(y,s+t),
$$
whenever $s,s+t\in(-4\eps,4\eps)$.
We will call $\Wman'$ a $4\eps$-flow-box at $z$, see Figure~\ref{fig:flowbox}.
\begin{figure}[ht]
\includegraphics[height=2.3cm]{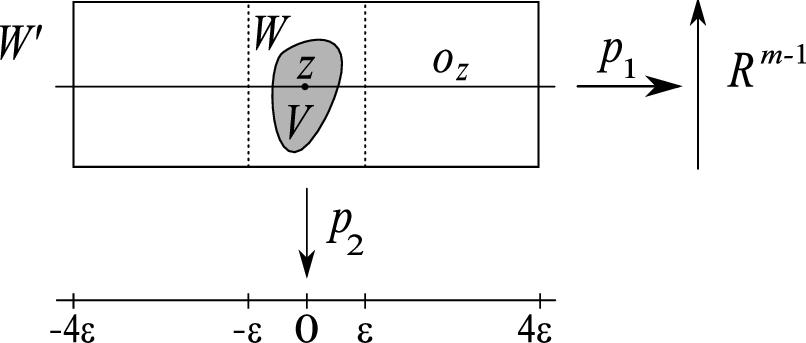} 
\caption{Flow box}
\protect\label{fig:flowbox}
\end{figure}
Let also $\Wman=\eta^{-1}\bigl(\RRR^{m-1}\times(-\eps,\eps)\bigr) \subset \Wman'$ be the ``central'' $\eps$-flow-box at $z$.
For every $\Dm$-submanifold $\Vman\subset\Wman$ denote
\begin{equation}\label{equ:UUU}
\fVW\;=\;\{\afunc\in\funcAV \ : \ \ShAV(\afunc)(\Vman) \subset \Wman \}.
\end{equation}
Then \ $\funcAWV \subset\fVW$ \ and
\begin{equation}\label{equ:ShAVU_imShAV_CVW}
\ShAV(\fVW) =\imShAV\cap \Ci{\Vman}{\Wman}
\end{equation}
is a $\Wr{0}$-open neighbourhood of $i_{\Vman}$ in $\imShAV$.

Let $p_1:\Wman'\to\RRR^{m-1}$ and $p_2:\Wman'\to(-4\eps,4\eps)$ be the standard projections.
Then we can define two maps $P_1:\fVW\to \Ci{\Vman}{\RRR^{m-1}}$ and $P_2:\fVW\to \Ci{\Vman}{(-\eps,\eps)}$ by
\begin{gather*}
P_1(\afunc) = p_1 \circ \ShAV(\afunc):\; \Vman \;\xrightarrow{~\ShAV(\afunc)~}\; \Wman\;\cong\;  \RRR^{m-1} \times (-\eps,\eps) \;\xrightarrow{~p_1~}\; \RRR^{m-1},
\\
P_2(\afunc) = p_2 \circ \ShAV(\afunc):\; \Vman \;\xrightarrow{~\ShAV(\afunc)~}\; \Wman\;\cong\;  \RRR^{m-1} \times (-\eps,\eps) \;\xrightarrow{~p_2~}\; (-\eps,\eps),
\end{gather*}
for $\afunc\in\fVW$.
Thus $\ShAV(\afunc)=(P_1(\afunc),P_2(\afunc))$.
\begin{theorem}\label{th:cond_open_imShAWV_reg}
Let $\Vman \subset\Wman$ be a \myemph{connected} compact $\Dm$-submanifold, and $\zer:\Vman\to\RRR$ be the zero function.
\begin{enumerate}
\item[\rm(1)] 
Then the map $P_1$ is locally constant with respect to the $\Wr{0}$ topology of $\fVW$ and its image $P_1(\fVW) \subset \Ci{\Vman}{\RRR^{m-1}}$ is at most countable.
\item[\rm(2)]
The shift map $\ShAV:\funcAV\to\imShAV$ is $\contWW{r}{r}$-open if and only if $P_1(\zer)$ is an isolated point of $P_1(\fVW)$ in $\Ci{\Vman}{\RRR^{m-1}}$ with respect to the $\Wr{r}$ topology.
\end{enumerate}
\end{theorem}
\begin{proof}
(1)
We need the following lemma:
\begin{lemma}\label{lm:main-prop-fVW}
Let $\afunc,\bfunc\in\fVW$.
Then one of the following conditions holds true:
\begin{enumerate}
 \item[\rm(i)]
$|\afunc(x)-\bfunc(x)|<2\eps$ for all $x\in\Vman$,
 \item[\rm(ii)]
$\afunc(x)-\bfunc(x)>6\eps$ for all $x\in\Vman$.
 \item[\rm(iii)]
$\bfunc(x)-\afunc(x)>6\eps$ for all $x\in\Vman$.
\end{enumerate}
Moreover, condition {\rm(i)} implies that 
\begin{enumerate}
 \item[\rm(iv)]
$P_1(\bfunc)=P_1(\afunc)$ \  and \
$P_2(\bfunc)=P_2(\afunc) + \bfunc - \afunc$.
\end{enumerate}
\end{lemma}
\begin{proof}
Define the following three open, mutually disjoint subsets of $\Vman$:
 \begin{itemize}
  \item[]
 $K_{1}=\{x\in\Vman \ : \ |\afunc(x)-\bfunc(x)|<2\eps \}$,
  \item[]
 $K_{2}=\{x\in\Vman \ : \ \afunc(x)-\bfunc(x)>6\eps \}$,
  \item[]
 $K_{3}=\{x\in\Vman \ : \ \bfunc(x)-\afunc(x)>6\eps \}$.
\end{itemize}
Let $x\in\Vman$ and suppose that $|\bfunc(x)-\afunc(x)|\geq2\eps$, i.e. $x\not\in K_1$.
Since $\Wman$ is the central $\eps$-flow-box of a $4\eps$-flow-box, and $$\ShAV(\afunc)(\Vman)\,\cup\,\ShAV(\bfunc)(\Vman)\,\subset\,\Wman,$$ we see that $|\bfunc(x)-\afunc(x)|>6\eps$, i.e.\! $x\in K_2 \cup K_3$.
Thus $\Vman=K_1\cup K_2 \cup K_3$.
Since $\Vman$ is connected and $K_i$ are open and disjoint, we obtain that $\Vman$ coincides with one of them.

(i)$\Rightarrow$(iv)
Denote $\amap=\ShAV(\afunc)$, $\bmap=\ShAV(\bfunc)$, and $\tau(x)=\bfunc(x)-\afunc(x)$.
Then 
\begin{equation}\label{equ:g_AFlow_f_tau}
\bmap(x) = \AFlow(x,\bfunc(x)) = \AFlow\bigl(\AFlow(x,\afunc(x)),\bfunc(x)-\afunc(x)\bigr) = \AFlow\bigl(\amap(x), \tau(x)\bigr)
\end{equation}
for all $x\in\Vman$.
Suppose that $|\tau|<2\eps$ on all of $\Vman$.
Recall that $\AFlow(y,s;t) = (y,s+t)$ whenever $s,t,s+t\in(-4\eps,4\eps)$ and $y\in\RRR^{m-1}$.
Since $|P_2(\afunc)|,|P_2(\bfunc)|<\eps$, we obtain that $|P_2(\afunc)+\tau|<3\eps$ and 
$$
(P_1(\bfunc),P_2(\bfunc)) = \bmap \stackrel{\eqref{equ:g_AFlow_f_tau}}{=\!=\!=}
\AFlow(P_1(\afunc),P_2(\afunc); \tau)=
(P_1(\afunc),P_2(\afunc)+ \tau).
$$
Hence 
$P_1(\bfunc)=P_1(\afunc)$ \  and \
$P_2(\bfunc)=P_2(\afunc) + \bfunc - \afunc$.
\end{proof}

\begin{corollary}\label{cor:classes_sim}
Define the following \myemph{equivalence\/} relation on $\fVW$ by 
$$\afunc\sim\bfunc \qquad \text{if and only if} \qquad |\afunc-\bfunc|<2\eps.$$
Then every class is $\Wr{0}$-open and by {\rm (iv)} of Lemma~\ref{lm:main-prop-fVW} the map $P_1$ is defined on the equivalence classes, whence $P_1$ is locally constant.

Moreover,  there is a well-defined strict order on the equivalence classes of $\sim$: if $A$ and $B$ are two distinct classes of $\sim$ then
$$
A > B
\qquad \text{if and only if} \qquad 
\afunc-\bfunc>6\eps
$$
on $\Vman$ for some $\afunc\in A$ and $\bfunc\in B$.
Hence there are at most countable many classes of $\sim$, and therefore the image of $P_1$ is at most countable.
 \end{corollary}
\begin{proof}
We have only to show that the definition of ``$>$'' does not depend on a particular choice of $\afunc$ and $\bfunc$.
Let $\afunc'\in A$ and $\bfunc'\in B$ be another functions.
Then 
$$
\afunc'-\bfunc' =
(\afunc'-\afunc) + (\afunc-\bfunc) + (\bfunc-\bfunc') > 
-2\eps + 6\eps - 2\eps =2\eps.
$$
Hence by Lemma~\ref{lm:main-prop-fVW}, $\afunc'-\bfunc' > 6\eps$ as well.
\end{proof}

This corollary proves (1).

(2) Let $A_{0}$ be the equivalence class of $\sim$ containing the zero function $\zer$, i.e. $A_{0}=\{ \afunc\in\fVW : |\afunc|<2\eps \}$.
Then it is easy to see that
$$
A_0 = \funcAWV, \qquad \text{whence} \qquad
\imShAWV=\ShAWV(A_0)=\ShAV(A_0).
$$
Thus we obtain the following commutative diagram
{\footnotesize
$$
 \xymatrix{ A_0 = \funcAWV  \  \ar[d]^{\ShAWV=\ShAV} \ar@{^{(}->}[r] & \quad\fVW\quad   \ar[d]^{\ShAV}  \ar@{^{(}->}[r] & \ \funcAV   \ar[d]^{\ShAV} \\
 \ShAV(A_0)=\imShAWV   \  \ar[d]_{p_1^{*}}   \ar@{^{(}->}[r] &\  \ShAV(\fVW)=\imShAV\cap\Ci{\Vman}{\Wman} \ \ar[d]^{p_1^{*}} \ar@{^{(}->}[r] & \ \imShAV \\
P_1(A_0)  \ \ar@{^{(}->}[r]     & \im P_1 & 
}
 $$}
\begin{lemma}\label{lm:imShAW_P10_isol}
The image $\imShAWV$ is $\Wr{r}$-open in $\imShAV$ for some $r\in\Nz$ if and only if $P_1(\zer)=P_1(A_0)$ is an isolated point of the image $\im P_1 =P_1(\fVW) \subset \Ci{\Vman}{\RRR^{m-1}}$ with respect to the $\Wr{r}$ topology.
\end{lemma}
\begin{proof}
\myemph{Sufficiency.}
Suppose that $P_1(\zer)$ is an isolated point of $\im P_1$ in the $\Wr{r}$ topology, i.e. $P_1(\zer)$ is a $\Wr{r}$-open subset of $\im P_1$.
Since $p_1^{*}$ is $\contWW{r}{r}$-continuous, we obtain that $\imShAWV =\ShAV(A_0)$ is $\Wr{r}$-open in $\ShAV(\fVW)$ being $\Wr{r}$-open in $\imShAV$.
Hence $\imShAWV$ is $\Wr{r}$-open in $\imShAV$ as well.

\myemph{Necessity.} 
Conversely, suppose that $P_1(\zer)$ is not isolated in $\im P_1$ with the $\Wr{r}$ topology.
We will construct a sequence of functions $\{\bfunc_i\}$ such that 
\begin{enumerate}
 \item[\rm(i)]
$P_1(\bfunc_i)\not= P_1(\zer)$,
 \item[\rm(ii)]
the sequence $\{\ShAV(\bfunc_i)\}$ converges to $i_{\Vman}:\Vman \subset\Mman$ in the $\Wr{r}$ topology.
\end{enumerate}
It will follow from (i) that $\ShAV(\bfunc_i)\not\in\imShAWV=\ShAV(A_0)$, and in particular $\bfunc_i\not\in A_0 = \funcAWV$ for all $i\in\NNN$.
Then from (ii) we will obtain that $\imShAWV=\ShAWV(A_0)$ is not a $\Wr{r}$-neighbourhood of $i_{\Vman}$ in $\imShAV$, i.e. $\imShAWV$ is not $\Wr{r}$-open in $\imShAV$.

Since $P_1(\zer)$ is not isolated in $\im P_1$, there exists a sequence of classes $\{A_i\}$ distinct from $A_0$ such that their images $P_1(A_i)$ converge to $P_1(A_0)$ in the $\Wr{r}$ topology.

It is easy to see that for every $i\in\NNN$ there exists a function $\bfunc_i\in A_i$ such that $P_2(\bfunc_i)(x,s)=P_2(\zer)(x,s)=s$.
Indeed, take any $\bfunc'_i\in A_i$ and set $\bfunc_i=\bfunc'_i- P_2(\bfunc'_i)$.

We claim that the sequence $\{\bfunc_i\}_{i\in\NNN}$ satisfies conditions (i) and (ii).
Property (i) holds since $P_1(\bfunc_i)=P_1(A_i)\not= P_1(A_0)$ for all $i\in\NNN$.

Moreover, we have that $P_2(\bfunc_i)=P_2(\zer)=p_2\circ i_{\Vman}$ and $\{P_1(\bfunc_i)\} \subset \Ci{\Vman}{\RRR^{m-1}}$ converges to $P_1(\zer)=p_1\circ i_{\Vman}$ in $\Wr{r}$ topology.
Hence $\ShAV(\bfunc_i)$ converges to $i_{\Vman}$ with respect to the $\Wr{r}$ topology.
This proves (ii).
\end{proof}

\begin{lemma}\label{lm:ShAWV_open}
The map $\ShAWV$ is $\contWW{r}{r}$-open  for all $r\geq0$. 
\end{lemma}
\begin{proof}
Notice that 
$$\funcAWV=\{\afunc\in \Ci{\Vman}{\RRR} \ : \ |\afunc(y,s)|<\eps, \ |s+\afunc(y,s)|<\eps  \},$$
and 
$$
\ShAWV(\afunc)(y,s) = (y,\afunc(y,s)+s).
$$
Hence $$\imShAWV=\{ \amap\in \Ci{\Vman}{\Wman} \ : \ p_1\circ \amap(y,s)=y, \ |p_2\circ \amap|<\eps\}.$$
and the inverse map $\ShAWV^{-1}:\imShAWV\to \funcAWV$ is given by
$$
\ShAWV^{-1}(\amap)(y,s) = p_2\circ\amap(y,s) - s.
$$
Evidently, this map is $\contWW{r}{r}$-continuous for all $r\in\Nzi$, whence $\ShAWV$ is $\contWW{r}{r}$-open.
\end{proof}

Now statement (2) follows from \ref{lm:imShAW_P10_isol}, \ref{lm:ShAWV_open}, and statement (B) of Theorem~\ref{th:charact_sh_open}.
Theorem~\ref{th:cond_open_imShAWV_reg} is completed.
\end{proof}

\subsection{Periodic case}\label{sect:recurrence_map}
Let $z$ be a periodic point of $\AFld$.
We will show that for a sufficiently small flow-box at $z$ \myemph{the images of $\sim$-classes defined in~\ref{cor:classes_sim} can be described in terms of its first recurrence map only}, see~\eqref{equ:PA_pRkd}.

Let $\theta$ be the period of $z$, $\trans$ be a codimension one open disk which transversally intersects the orbit of $\orb_{z}$ at $z$, and $R:(\trans,z)\to(\trans,z)$ be the first recurrence map.
Let also $\eps<\theta/5$ and $\eta:\Wman'\to\RRR^{m-1}\times(-4\eps,4\eps)$ be a $4\eps$-flow-box at $z$.
Decreasing $\Wman'$ we can assume, in addition, that 
\begin{equation}\label{equ:oz_Wprime}
\orb_{z}\cap\Wman'=\AFlow(z\times(-4\eps,4\eps)),
\end{equation}
$\trans\subset\Wman$, and $\trans$ is transversal to all orbits of $\Wman'$, so that the restriction $p_1|_{\trans}:\Wman'\supset\trans\to\RRR^{m-1}$ is a diffeomorphism.
Put 
$$
d= (p_1|_{\trans})^{-1}  \circ p_1 :\Vman\xrightarrow{~p_1~}\RRR^{m-1} \xrightarrow{~(p_1|_{\trans})^{-1}~} \trans.
$$
Then $d$ preserves the first coordinate.

Let $\Wman \subset\Wman'$ be the central $\eps$-flow-box, $\Vman\subset\Wman$ be a connected $\Dm$-neigh\-bourhood of $z$, and $A$ be the $\sim$-class of $\fVW$.
Then it follows from~\eqref{equ:oz_Wprime} that there exists a unique $k\in\ZZZ$ such that $|\afunc(z)-k\theta|<\eps$ for every $\afunc\in A$, and
\begin{equation}\label{equ:PA_pRkd}
P_1(A) = p_1 \circ R^k \circ d:\Vman\xrightarrow{~d~} \trans\xrightarrow{~R^k~}\trans\xrightarrow{~p_1~}\RRR^{m-1}.
\end{equation}

\subsection{Proof of Theorem~\ref{th:ShAVopen_reg}.}\label{sect:proof_th:ShAVopen_reg}
Due to Theorem~\ref{th:cond_open_imShAWV_reg} it suffices to find a $4\eps$-flow-box neighbourhood $\Wman'$ of $z$ such that for every connected $\Dm$-neighbourhood $\Vman$ of $z$, contained in the central $\eps$-flow-box neighbourhood $\Wman$, the image $P_1(\zer)$ is an isolated point of $P_1(\fVW)$ in the $\Wr{r}$ topology, where $r\geq0$ in the cases (a) and (b), and $r\geq1$ in the case (c).

(a) Suppose that $z$ is non-recurrent.
Then there exist $\eps>0$, and a $4\eps$-flow-box neighbourhood $\Wman'$ of $z$ such that $\AFlow(z,t)\in\Wman'$ iff $|t|<4\eps$.
Let $\Wman\subset\Wman'$ be the central $\eps$-flow-box and $\Vman\subset\Wman$ be a connected $\Dm$-neighbourhood.
Then it follows from Lemma~\ref{lm:main-prop-fVW} that there exists only one equivalence class of $\sim$.
Hence $\im P_1 = P_1(\zer)$, and thus $P_1(\zer)$ is isolated in $\im P_1$ in any of $\Wr{r}$ topologies.

Suppose that $z$ is periodic.
Let $\trans$ be a codimension one open disk which transversally intersects the orbit of $\orb_{z}$ at $z$, $R:(\trans,z)\to(\trans,z)$ be the first recurrence map, and $\Wman'$, $\Wman$, and $\Vman$ be such as in \S\ref{sect:recurrence_map}.

(b) If the germ of $R$ at $z$ is periodic of some period $q$, then we can assume that there exists an open neighbourhood $E \subset \trans$ of $z$ such that $R(E)=E$ and $d(\Vman)\subset E$.
Then it follows from~\eqref{equ:PA_pRkd} that the image of $P_1$ is finite and consists of $q$ points.
Hence $P_1(\zer)$ is isolated in $\im P_1$ in any of $\Wr{r}$ topologies.

(c)
Suppose that the tangent map $T_z R: T_z \trans\to T_z \trans$ has at least one non-zero eigen vector $v\in T_{z} \trans \subset T_{z}\Vman$ with eigen value $\lambda$ such that $|\lambda|\not=1$.
Then $T_{z}R^k(v) = \lambda^{k} v$.
Let $T_{z}p_1: T_{z} \trans\to T_{0}\RRR^{m-1}$ be the tangent map of $p_1$ at $z$ and $u=T_{z}p_1(v)$.
Then for every class $A$ corresponding to some $k\in\ZZZ$ we have that 
$$T_z(P_1(A)) \stackrel{~\eqref{equ:PA_pRkd}~}{=\!=\!=} T_z(p_1 \circ R^k \circ d)(v) = \lambda^{k} u.$$
Since $|\lambda|\not=1$, we see that $P(\zer)$ is isolated in $\im P_1$ in any of $\Wr{r}$ topologies for $r\geq1$.
\hfill \hfill \qed

\section{Openness of $\imShAWV$ in $\imShA$. Singular case}\label{sect:open_im_in_im_fixpt}
Let $\Wman\subset\Mman$ be an open subset and $\Vman\subset\Wman$ be a $\Dm$-submanifold.
Similarly to~\eqref{equ:UUU} put
\begin{equation}\label{equ:UUU1}
\fVW\;=\;\{\afunc\in\funcAV \ : \ \ShAV(\afunc)(\Vman) \subset \Wman \}.
\end{equation}
Then \ $\funcAWV \subset\fVW$ \ and 
\begin{equation}\label{equ:ShAVU_imShAV_CVW1}
\ShAV(\fVW) =\imShAV\cap \Ci{\Vman}{\Wman}
\end{equation}
is a $\Wr{0}$-open neighbourhood of $i_{\Vman}$ in $\imShAV$.
We will now present a sufficient condition which guarantees that $\funcAWV =\fVW$.
Due to~\eqref{equ:ShAVU_imShAV_CVW1} this will imply that 
\begin{equation}\label{equ:imShAWV_open_in_imShAV}
\imShAWV\defeq\ShAV(\funcAV)=\ShAV(\fVW)=\imShAV\cap \Ci{\Vman}{\Wman}
\end{equation}
is $\Wr{0}$-open in $\imShAV$.
In particular we prove~\eqref{equ:imShAWV_open_in_imShAV} for vector fields of types $\typeZero$, $\typeLinear$, $\typeHamilt$ and their products.

\begin{definition}\label{defn:prQ}
Let $\Vman\subset\Wman$ be a subset.
We will say that a pair $(\Wman,\Vman)$ has \myemph{proper boundary intersection} property (\prQ) if for any $x\in\Vman$, $a\in\RRR$, and $T>1$ such that 
\begin{itemize}
 \item $\AFlow(x,\tau a) \in \Wman$ for $\tau\in[0,1) \cup\{T\}$, but 
 \item $\AFlow(x,a) \in \Fr(\Wman)=\overline{\Wman}\setminus\Wman$,
\end{itemize}
there exists $\tau'\in(1,T)$ such that $\AFlow(x,\tau' a) \not\in\overline{\Wman}$, see Figure~\ref{fig:bi_prop}a).
\end{definition}
Roughly speaking, if the orbit $\orb_x$ of any $x\in\Vman$ leaves $\Wman$ for a certain amount of time and comes back into $\Wman$, then during this time it must leave the closure $\overline{\Wman}$.
In Figure~\ref{fig:bi_prop}b) the pair $(\Wman,\Vman)$ does not satisfy \prQ
\begin{center}
\begin{figure}[ht]
\begin{tabular}{ccc}
\includegraphics[height=2cm]{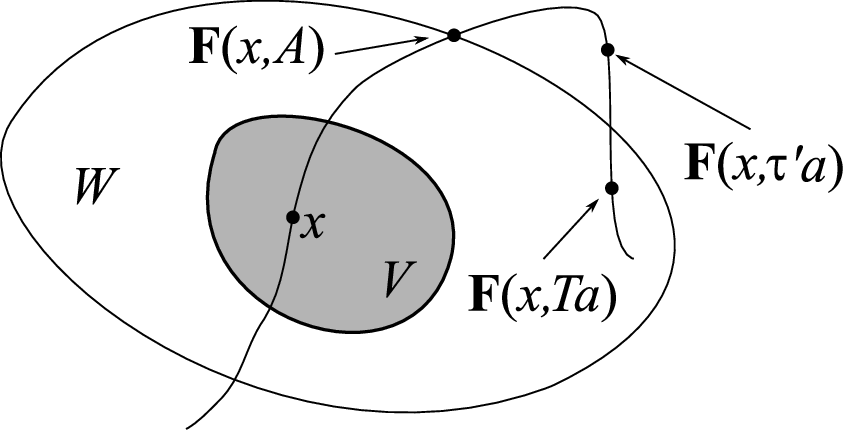} &
\qquad \qquad &
\includegraphics[height=2cm]{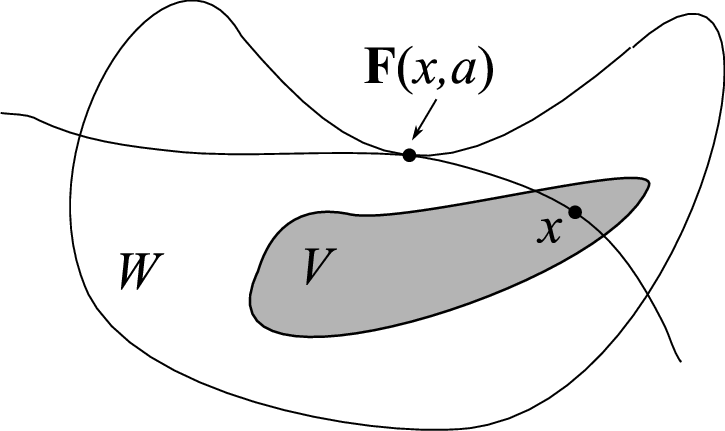}  \\
a) & & b)
\end{tabular}
\caption{}
\protect\label{fig:bi_prop}
\end{figure}
\end{center}
The following lemma is the crucial implication of \prQ:
\begin{lemma}\label{lm:prQ_singpt}
Let $z\in\Wman$, $\AFld(z)=0$, and $\Vman \subset\Wman$ be a connected $\Dm$-neigh\-bourhood of $z$ such that $(\Wman,\Vman)$ has \prQ\
Then $\funcAWV=\fVW$.
\end{lemma}
\begin{proof}
Let $\afunc\in\fVW$, so $\AFlow(x, \afunc(x)) \in \Wman$ for all $x\in\Wman$.
We have to show that $\afunc\in\funcAWV$.
It suffices to verify that the following subset of $\Vman$:
$$\Ka = \{ x\in\Vman \, : \, \AFlow(x, s\afunc(x)) \in \Wman \ \text{for all} \ s\in I=[0,1] \}$$
coincides with all of $\Vman$.
Notice that $z\in\Ka$.
Moreover, since $\Wman$ is open in $\Mman$, it follows that $\Ka$ is open in $\Vman$.
Therefore it suffices to show that $\Vman\setminus\Ka$ is open in $\Vman$ as well.
Since $\Vman$ is connected, we will get that $\Ka=\Vman$.

Let $x\in\Vman\setminus\Ka$.
Then $x=\AFlow(x,0\cdot\afunc(x))$ and $\AFlow(x,1\cdot\afunc(x))$ belong to $\Wman$ and there exists $\tau_0\in(0,1)$ such that $\AFlow(x,\tau\,\afunc(x))\in\Wman$ for all $\tau\in[0,\tau_0)$ while $\AFlow(x,\tau_0\,\afunc(x))\in\Fr(\Wman)$.
Now it follows from \prQ\ for the pair $(\Wman,\Vman)$ that there exists $\tau'\in(\tau_0,1)$ such that $\AFlow(x,\tau'\,\afunc(x))\not\in\overline{\Wman}$.

Then there is an open neighbourhood $\Vman_{x}$ of $x$ in $\Vman$ such that 
$$\AFlow(y,\tau'\,\afunc(y))\not\in\overline{\Wman}$$ for all $y\in\Vman_{x}$.
Hence $\Vman_x\subset\Vman\setminus\Ka$, and thus $\Vman\setminus\Ka$ is open in $\Vman$.
\end{proof}

The following simple lemma is left for the reader.
\begin{lemma}\label{lm:prop_prQ}
{\rm (1)}~If $(\Wman,\Vman)$ has \prQ, then for every subset $\Vman'\subset\Vman$ the pair $(\Wman,\Vman')$ also has this property.

{\rm (2)}~Let $\Wman'$ be an open neighbourhood of $\overline{\Wman}$.
Then $(\Wman,\Vman)$ has property \prQ\ with respect to $\AFld$ iff it has this property with respect to the restriction $\AFld|_{\Wman'}$.
In other words, \prQ\ is determined by the behavior of $\AFld$ on arbitrary small neighbourhood of $\overline{\Wman}$ only.

{\rm (3)}~For every $x\in\Wman$ denote by $\gamma_x$ the connected component of $\orb_x\cap\Wman$ containing $x$.
Suppose that $\Fr(\Wman)$ is a smooth submanifold of $\Mman$ and for every $x\in\Vman$ its orbit $\orb_x$ is transversal to $\Fr(\Wman)$ at each $y\in\overline{\gamma_x}\cap\Fr(\Wman)$ whenever such a point exists.
Then $(\Wman,\Vman)$ has \prQ
\end{lemma}

\subsection{Isolating blocks.}
Statement (3) of Lemma~\ref{lm:prop_prQ} is relevant with one of the principal results of~\cite{ConleyEaston:TRAMS:1971}.
An open subset $\Uman\subset\Mman$ is an \myemph{isolating neighbourhood} for the flow $\AFlow$ if $\orb_x \not\subset\overline{\Uman}$ for all $x\in\Fr\Uman=\overline{\Uman}\setminus\Uman$.
A closed $\AFlow$-invariant subset $X\subset\Mman$ is \myemph{isolated}, if $X$ is the maximal invariant subset of some isolating neighbourhood $\Uman$ of $X$.

Let $\Wman$ be a compact $\Dm$-submanifold of $\Mman$ and $w=\overline{\Wman}\setminus\Int\Wman$ be its boundary.
Denote 
$$
\begin{array}{rcl}
w^{+} &=& \{x\in w \ : \ \exists\eps>0 \ \text{with} \ \AFlow(x\times(-\eps,0))\cap\Wman=\varnothing\}, \\
w^{-} &=& \{x\in w \ : \ \exists\eps>0 \ \text{with} \ \AFlow(x\times(0,\eps))\cap\Wman=\varnothing\}, \\
\tau  &=& \{x\in w \ : \ \text{$\AFld$ is tangent to $w$}\}.
\end{array}
$$
Then $\Wman$ is called an \myemph{isolating block\/} for $\AFlow$ if $w^{+}\cap w^{-}=\tau$ and $\tau$ is a smooth submanifold of $w$ with codimension one.
Thus $w^{+}$ and $w^{-}$ are submanifolds (possibly with corners is so is $\Wman$) of $w$ with common boundary $\tau$.
It follows that the interior of an isolating block is an isolating neighbourhood.
Moreover, from (3) of Lemma~\ref{lm:prop_prQ} we obtain
\begin{lemma}\label{lm:isol_block}
If $\Wman$ is an isolating block, then for any subset $\Vman\subset\Int\Wman$ the pair $(\Int\Wman,\Vman)$ has \prQ
\end{lemma}

\begin{theorem}{\rm\cite{ConleyEaston:TRAMS:1971}}\label{th:ConleyEaston:TRAMS:1971}
Let $X \subset\Mman$ be a closed isolated $\AFlow$ invariant subset and $\Uman\supset X$ be its isolating neighbourhood.
Then there exists an isolating block $\Wman$ such that $X\subset\Int\Wman\subset\Uman$.
\end{theorem}
This result was established for the case when $\Wman$ is a manifold with boundary, but if $\partial\Mman$ is $\AFlow$-invariant and $X\cap\partial\Mman\not=\varnothing$, then the proof easily extends to $\Dm$-submanifolds.

\begin{corollary}\label{cor:prQ_for_isolated_pt}
Suppose $z$ is an isolated singular point of $\AFld$.
Then there exist a base $\beta=\{\Wman_{\alpha}\}_{\alpha\in A}$ of open neighbourhoods of $0\in\RRR^{k}$ such that for every $\Wman\in\beta$ and any subset $\Vman\subset\Wman$ the pair $(\Wman,\Vman)$ has \prQ.
\end{corollary}

\subsection{Semi-invariant sets.}
We say that $\Wman$ is \myemph{negatively} (\myemph{positively}) invariant with respect to $\AFld$ if for every $x\in\Wman$ and $t\leq0$ ($t\geq0$) such that $(x,t)\in\domA$ we have $\AFlow(x,t)\in\Wman$.
In this case we will also say that $\Wman$ is \myemph{semi-invariant}.
\begin{lemma}\label{lm:semiinvar_prQ}
If $\Wman$ is semi-invariant with respect to $\AFlow$, then for any subset $\Vman\subset\Wman$ the pair $(\Wman,\Vman)$ has \prQ
\end{lemma}
\begin{proof}
Notice that if $x,\AFlow(x,T a)\in\Wman$ for some $a\in\RRR$ and $T>1$, then it follows from semi-invariance of $\Wman$ that $\AFlow(x,\tau a)\in\Wman$ for all $\tau \in[0,T]$.
Hence there are no $x\in\Vman$, $a\in\RRR$ and $T>1$ satisfying assumptions of Definition~\ref{defn:prQ}.
Therefore $(\Wman,\Vman)$ has \prQ
\end{proof}

\subsection{Product of flows.}
For $i=1,2$ let $\Mman_i$ be a manifold, $\AFld_i$ be a vector field on $\Mman_i$, $\Wman_i\subset\Mman_i$ be an open subset, and $\Vman_i \subset\Wman_i$ be a subset.
Denote $\Mman=\Mman_1\times\Mman_2$ and $\Wman=\Wman_1\times\Wman_2$.
Consider the product of these vector fields $\AFld(x,y)=(\AFld_1(x), \AFld_2(y))$ on $\Mman$.
It generates a local flow $\AFlow(x,y,t)=(\AFlow_1(x,t),\AFlow_2(y,t))$.
\begin{lemma}\label{lm:prQ_linflows}
Suppose that $(\Wman_i,\Vman_i)$ has \prQ\ with respect to $\AFld_i$, $(i=1,2)$, and let $\Vman\subset\Vman_1\times\Vman_2$ be a subset.
Then $(\Wman,\Vman)$ has \prQ\ with respect to $\AFld$.
\end{lemma}
\begin{proof}
Let $x=(x_1,x_2)\in\Vman$, $a\in\RRR$, and $T>1$ be such that 
$$\AFlow(x,\tau a)=(\AFlow_1(x_1,\tau a),\AFlow_2(x_2,\tau a)) \in \Wman, \qquad \tau\in[0,1) \cup\{T\},$$
$$
\AFlow(x,a)=(\AFlow_1(x_1,a),\AFlow_2(x_2,a)) \in \Fr{\Wman}=(\Fr(\Wman_1)\times\Wman_2) \cup (\Wman_1\times\Fr(\Wman_2)).
$$
In particular, $\AFlow_i(x_i,a)\in\Fr(\Wman_i)$ for at least one index $i=1,2$.
For definiteness assume that $\AFlow_1(x_1,a)\in\Fr(\Wman_1)$.
Since $x_1\in\Vman_1$, it follows from \prQ\ for $(\Wman_1,\Vman_1)$ that there exists $\tau'\in(1,T)$ such that 
$\AFlow_1(x_1,\tau' a) \not\in \overline{\Wman_1}$.
Then $\AFlow(x,\tau' a) \not\in\overline{\Wman}$ as well.
Hence $(\Wman,\Vman)$ has \prQ\ with respect to $\AFld$.
\end{proof}

\begin{corollary}\label{cor:FF_imShAWV_open}
Let $\AFld$ be a product of vector fields of types $\typeLinear$ or $\typeHamilt$ on $\RRR^{m}$.
Then there exist a base $\beta=\{\Wman_{\alpha}\}_{\alpha\in A}$ of open neighbourhoods of $0\in\RRR^{k}$ such that for every $\Wman\in\beta$ and any subset $\Vman\subset\Wman$ the pair $(\Wman,\Vman)$ has \prQ\
In particular, if $\Vman$ is a connected $\Dm$-neighbourhood of $0$, then $\funcAWV =\fVW$ and thus by~\eqref{equ:ShAVU_imShAV_CVW1} $$\imShAWV=\imShAV\cap \Ci{\Vman}{\Wman}$$ is $\Wr{0}$-open in $\imShAV$.
\end{corollary}
\begin{proof}
Due to Lemma~\ref{lm:prQ_linflows} it suffices to assume that $\AFld$ is of type $\typeLinear$ or $\typeHamilt$.
If $\AFld$ is of type $\typeHamiltExtr$, then $0$ has an arbitrary small invariant neighbourhoods $\Wman$ and by Lemma~\ref{lm:semiinvar_prQ} $(\Wman,\Vman)$ has \prQ\ for any subset $\Vman\subset\Wman$.
If $\AFld$ is of type $\typeHamiltNonExtr$ then $\{0\}\subset\RRR^2$ is an isolated invariant subset of any of its small neighbourhood, whence existence of $\Wman$ follows from Theorem~\ref{th:ConleyEaston:TRAMS:1971}, though it can easily be constructed without referring to this theorem.

It remains to consider the case when $\AFld(x)=Bx$, where $B$ is a Jordan cell corresponding either to some eigen value $a\in\RRR$ or to the pair of complex conjugate $a\pm ib$, $(a,b\in\RRR)$.

1) If $a>0$ $(a<0)$, then, e.g.\cite{PalisdeMelo:1982}, $0$ has arbitrary small positively (negatively) $\AFlow$-invariant neighbourhoods $\Wman$.
Moreover, if $B = \left\|\begin{smallmatrix} 0  & -b \\ b  & 0 \end{smallmatrix}\right\|$, then $0\in\RRR^2$ has arbitrary small $\AFlow$-invariant neighbourhoods $\Wman$.
Then again by Lemma~\ref{lm:semiinvar_prQ} such neighbourhoods have desired properties.

2) Suppose that 
$B= \left\|\begin{smallmatrix} 
0 & 1  \\
    & \cdots &  \\
&   0 & 1 \\
&     & 0 \\
 \end{smallmatrix}\right\|
$ is nilpotent.
Then the coordinate functions of $\AFlow$ are given by the following formulas:
$$
\AFlow_1(x,t) =x_1 + x_2 t + x_3 \frac{t^2}{2!} + \cdots + x_m \frac{t^{m-1}}{(m-1)!},
\qquad 
\AFlow_i = \frac{\partial \AFlow_{i-1}}{\partial t},
$$
for $i=2,\ldots,m$.
In particular, for each $x\in\RRR^{m}$ these functions are polynomials in $t$ of degree $\leq m-1$.
Hence for every $\bar r>0$ there exists $T>0$ such that if $|x_0|<\bar r$ and $x_0$ is not a fixed point of $\AFlow$, then $\|\AFlow(x,t)\|$ strictly monotone increases when increases $t>T$ (decreases $t<-T$).

For every $r\geq0$ let $\Vman_r \subset\RRR^{m}$ be the closed $m$-disk of radius $r$ centered at the origin and $\Wman_r=\Int(\Vman_r)$ be its interior.
Since $\AFlow(0,t)=0$ for all $t\in\RRR$, it follows that for every $R>0$ there exists $r<\bar r$ such that $\AFlow(\Vman_{r}\times[-T,T]) \subset \Wman_R$, i.e. $\|\AFlow(x,t)\|<R$ for $\|x\|\leq r$ and $|t|\leq T$.
We claim that the pair $(\Wman_R,\Vman_r)$ has \prQ

Suppose that for some $x\in \Vman_r$ and $a\in\RRR$ we have $\AFlow(x,\tau a)\in \Wman_R$ for $\tau\in[0,1)$ and  $\AFlow(x,a)\in \Fr(\Wman_R)=\partial \Vman_R$, i.e. $\|\AFlow(x,\tau a)\|<R$ and $\|\AFlow(x,a)\|=R$.
Then $|a|>T$, whence $\|\AFlow(x,t)\|$ strictly monotone increases when increases $|t|$.
In particular, $\|\AFlow(x,\tau a)\|>R$ for all $\tau>1$, i.e. $\AFlow(x,\tau a)\not\in\Vman_R=\overline{\Wman_R}$.

3) Suppose that $\lambda=ib$, $(b\not=0)$, is purely imagine but $m=2k\geq4$ for some $k\geq 2$.
Then regarding $\RRR^{m}$ as $\CCC^{k}$ we have that in complex coordinates  
$A= \left\|\begin{smallmatrix} 
ib & 1  \\
    & \cdots &  \\
&   ib & 1 \\
&     & ib \\
 \end{smallmatrix}\right\|
$.
Hence the coordinate functions of $\AFlow$ are given by formulae similar to the case 2).
Denote $z=(z_1,\ldots,z_k)$ and $p(z,t)=z_1 + z_2 t + z_3 \frac{t^2}{2!} + \cdots + z_k \frac{t^{k-1}}{(k-1)!}$.
Then
$$
\AFlow_1(z,t) = e^{ib} p(z), \qquad
\AFlow_i = e^{ib} \frac{\partial \AFlow_{i-1}}{\partial t}, \qquad i=2,\ldots,k.
$$
Since $|e^{ib}|=1$, it follows that $\|\AFlow(x,t)\|$ satisfies monotonicity conditions analogous to the case 2).
Then by the similar arguments we obtain that for every $R>0$ there exists $r>0$ such that the pair $(\Wman_R,\Vman_r)$ has property \prQ
\end{proof}

\section{Proof of Theorem~\ref{th:main-result}}\label{sect:proof_th:main-result}
Let $\AFld$ be a vector field of class $\FF(\Mman)$.
It follows from results of~\cite{Maks:ImSh} that $\imShA=\EidAFlow{k}$, where $k=1$ if $\FixF\not=\FixLinE\cup\FixHamNonExtrE$ and $k=0$ otherwise.

Therefore we should prove that the map $\ShA:\funcA\to\imShA$ is a homeomorphism with respect to $\Sr{\infty}$ topologies.
Actually we want to apply Theorem~\ref{th:Sh-open-map}.
\begin{clm}\label{clm:ShAV_open_FF}
Suppose that $\AFld$ belongs to class $\FF(\Mman)$. 
\begin{enumerate}
 \item[\rm(1)]
If $z\in\FixF$, then for any sufficiently small connected compact $\Dm$-neigh\-bourhood $\Vman$ of $z$ the shift map $\ShAV:\funcAV\to\imShAV$ is $\contWW{\infty}{\infty}$-open.
Moreover, if $z\in\FixLinPHE\cup\FixLinNilpE$, then $\ShAV$ is even $\contWW{r}{r+2}$-open for all $r\geq0$.
 \item[\rm(2)] 
If $z$ is a regular point of $\AFld$, then for any sufficiently small connected compact $\Dm$-neighbourhood $\Vman$ of $z$ the shift mapping $\ShAV:\funcAV\to\imShAV$ is $\contWW{r}{r}$-open for all $r\geq1$.
\end{enumerate}
\end{clm}
\begin{proof}
By (C) Theorem~\ref{th:charact_sh_open} it suffices to find a neighbourhood $\Wman$ of $z$ such that for every connected $\Dm$-neighbourhood $\Vman\subset\Wman$ of $z$
\begin{enumerate}
\item[\rm(i)]
the shift map $\ShAWV:\funcAWV\to\imShAWV$ is $\contWW{r}{s}$-open for the corresponding values $r,s$, and
 \item[\rm(ii)] 
its image $\imShAWV$ is $\Wr{k}$-open in $\imShA$, where $k=0$ in the cases (1), and $k=1$ in the case (2).
\end{enumerate}

(1) Let $z\in\FixA$.
Then by definition of class $\FF(\Mman)$ there exists a neighbourhood $\Wman$ of $z$ on which in some local coordinates either 
\begin{enumerate}
 \item[(a)]
$\AFld$ is a product of finitely many vector fields $\BFld_1,\ldots,\BFld_n$ each of which is of types $\typeZero$, $\typeLinear$, or $\typeHamilt$, or
 \item[(b)]
$\AFld$ belongs to one of the types $\typeLinPH'$, $\typeLinNilp'$, $\typeLinRotExt'$, or $\typeHamilt'$ but $z$ is an isolated singular point of $\AFld$.
\end{enumerate}

Decreasing $\Wman$ we can also assume that $\imShAWV$ is $\Wr{0}$-open in $\imShA$, i.e.\! condition (ii) is satisfied.
In the case (a) this follows from Corollary~\ref{cor:FF_imShAWV_open}, while in the case (b) from Corollary\;\ref{cor:prQ_for_isolated_pt}.

Then (i) directly follows from Lemmas\;\ref{lm:reparam_shift_map}-\ref{lm:lin_flows}.

(2) Suppose that $z$ is a regular point for $\AFld$.
Since $\AFld$ belongs to $\FF(\Mman)$, it follows that the assumptions of Theorem~\ref{th:ShAVopen_reg} are satisfied, whence for any sufficiently small $\Dm$-neighbourhood of $z$ the shift map $\ShAV$ is $\contWW{r}{r}$-open for all $r\geq1$.
\end{proof}

For every $z\in\Mman$ let $\Vman_{z}$ be a neighbourhood guaranteed by Claim~\ref{clm:ShAV_open_FF}.
By assumption of Theorem~\ref{th:main-result} the set $\overline{\FixF\setminus(\FixLinPHE\cup\FixLinNilpE)}$ is compact.

Therefore using paracompactness of $\Mman$ we can find a locally finite cover $\{\Vman_i\}_{i\in\Lambda}$ of $\Mman$ by compact connected $\Dm$-submanifolds and a finite subset $\Lambda'\subset\Lambda$ such that the map $\ShAVi$ is $\contWW{\infty}{\infty}$-open for $i\in\Lambda'$ and $\contWW{r}{r+2}$-open for all $r\geq0$ whenever $i\in\Lambda\setminus\Lambda'$.
Hence by ``fragmentation'' Theorem~\ref{th:Sh-open-map} $\ShA$ is a local homeomorphism with respect to $\Sr{\infty}$ topologies.
All other statements concerning homotopy types of $\DidAFlow{r}$ and $\EidAFlow{r}$ follow by the arguments of the proof of~\cite[\S9]{Maks:TA:2003}.
Theorem~\ref{th:main-result} is completed.

\section{Acknowledgements}
I would like to thank V.\;V.\;Sharko, A.\;Prishlyak, I.\;Vlasenko, and I.\;Yurchuk for useful discussions.
I am especialy thank E.\;Polulyakh for very helpful conversations and pointing out to the irrational flow on $T^2$ which lead me to discovering mistakes of~\cite{Maks:TA:2003}.

\def\cprime{$'$}
\providecommand{\bysame}{\leavevmode\hbox to3em{\hrulefill}\thinspace}
\providecommand{\MR}{\relax\ifhmode\unskip\space\fi MR }
\providecommand{\MRhref}[2]{%
  \href{http://www.ams.org/mathscinet-getitem?mr=#1}{#2}
}
\providecommand{\href}[2]{#2}


\end{document}